\documentclass[12pt]{amsart}
\usepackage{stmaryrd}
\usepackage{mathrsfs}
\usepackage{eucal}
\usepackage[all]{xy}
\usepackage[margin=1in]{geometry} 
\usepackage[bookmarks, bookmarksdepth=2, colorlinks=true, linkcolor=blue, citecolor=blue, urlcolor=blue]{hyperref}
\usepackage{xcolor}
\usepackage[shortlabels]{enumitem}

\newcommand{\comment}[1]{}
\newcommand{\bA}{\mathbf{A}}

\newcommand{\sL}{\mathscr{L}}

\newcommand{\bN}{\mathbf{N}}

\newcommand{\cO}{\mathcal{O}}

\newcommand{\bQ}{\mathbf{Q}}

\newcommand{\bS}{\mathbf{S}}

\newcommand{\bV}{\mathbf{V}}

\newcommand{\fp}{\mathfrak{p}}

\newcommand{\fq}{\mathfrak{q}}

\makeatletter
\@addtoreset{equation}{section}
\makeatother

\numberwithin{equation}{section}
\newtheorem{theorem}[equation]{Theorem}

\newtheorem{proposition}[equation]{Proposition}
\newtheorem{lemma}[equation]{Lemma}
\newtheorem{corollary}[equation]{Corollary}

\theoremstyle{definition}
\newtheorem{rmk}[equation]{Remark}
\newenvironment{remark}[1][]{\begin{rmk}[#1] \pushQED{\qed}}{\popQED \end{rmk}}
\newtheorem{eg}[equation]{Example}
\newenvironment{example}[1][]{\begin{eg}[#1] \pushQED{\qed}}{\popQED \end{eg}}
\newtheorem{defnaux}[equation]{Definition}
\newenvironment{definition}[1][]{\begin{defnaux}[#1]\pushQED{\qed}}{\popQED \end{defnaux}}
\newtheorem{constr}[equation]{Construction}

\newcommand{\stacks}[1]{\cite[\href{http://stacks.math.columbia.edu/tag/#1}{Tag~#1}]{stacks}}
\newcommand{\DOI}[1]{\href{http://doi.org/#1}{\color{purple}{\tiny\tt DOI:#1}}}
\newcommand{\arxiv}[1]{\href{http://arxiv.org/abs/#1}{{\tiny\tt arXiv:#1}}}
\newcommand{\defn}[1]{\textit{#1}}

\let\ol\overline
\let\ul\underline
\let\lbb\llbracket
\let\rbb\rrbracket
\renewcommand{\phi}{\varphi}
\DeclareMathOperator{\rank}{rk}
\DeclareMathOperator{\brank}{brk}
\DeclareMathOperator{\im}{im} 
\DeclareMathOperator{\Sym}{Sym}
\DeclareMathOperator{\Spec}{Spec}

\newcommand{\id}{\mathrm{id}}
\newcommand{\pol}{\mathrm{pol}}
\DeclareMathOperator{\Frac}{Frac}
\DeclareMathOperator{\Proj}{Proj}
\newcommand{\GL}{\mathbf{GL}}

\DeclareMathOperator{\Sh}{Sh}

\newcommand{\umu}{\ul{\smash{\mu}}} 

\newcommand{\ulambda}{\ul{\smash{\lambda}}}

\renewcommand{\int}{\operatorname{int}}
\newcommand{\lpp}{(\!(}
\newcommand{\rpp}{)\!)}

\title{Uniformity for limits of tensors}

\author{Arthur Bik}
\address{Institute for Advanced Study, Princeton, NJ, USA, and MPI for Mathematics in the Sciences, Germany}
\email{\href{mailto:mabik@ias.edu}{mabik@ias.edu}}
\urladdr{\url{http://arthurbik.nl}}

\author{Jan Draisma}
\address{University of Bern, Switzerland, and Eindhoven University of Technology, The Netherlands}
\email{\href{mailto:jan.draisma@math.unibe.ch}{jan.draisma@math.unibe.ch}}
\urladdr{\url{https://mathsites.unibe.ch/jdraisma/}}

\author{Rob Eggermont}
\address{Eindhoven University of Technology, The Netherlands}
\email{\href{mailto:r.h.eggermont@tue.nl}{r.h.eggermont@tue.nl}}
\urladdr{\url{https://www.tue.nl/en/research/researchers/rob-eggermont/}}

\author{Andrew Snowden}
\address{Department of Mathematics, University of Michigan, Ann Arbor, MI, USA}
\email{\href{mailto:asnowden@umich.edu}{asnowden@umich.edu}}
\urladdr{\url{http://www-personal.umich.edu/~asnowden/}}

\thanks{AB was partially supported by Postdoc.Mobility Fellowship number P400P2\_199196 from the Swiss National Science
Foundation and a grant from the Simons Foundation (816048, LC).
JD was partially supported by the Vici grant 639.033.514 from the Netherlands Organisation for Scientific Research and by Swiss National Science Foundation project grant 200021\_191981.
RE was partially supported by the NWO Veni grant entitled {\em Stability and structure in infinite-dimensional spaces}, project number 016.Veni.192.113.
AS was supported by NSF grant DMS-1453893.
}

\date{December 20, 2023}

\begin{document}

\begin{abstract}
There are many notions of rank in multilinear algebra: tensor rank, partition rank, slice rank, and strength (or Schmidt rank) are a few examples. Typically the rank $\le r$ locus is not Zariski closed, and understanding the closure (the locus with \emph{border rank} $\le r$) is an important problem. We make two contributions in this direction: we prove a de-bordering result, which bounds border rank as a function of rank; and we show that the limits required to realize a point of border rank $\le r$ do not become increasingly complicated as the dimension of the vector space increases. We prove both results for a fairly general class of ranks. We deduce our theorems on ranks from foundational results on $\GL$-varieties, which are infinite dimensional algebraic varieties on which the infinite general linear group acts. For example, an important result concerns the existence of curves on $\GL$-varieties.
\end{abstract}

\maketitle
\tableofcontents

\section{Introduction}

Let $V$ be a finite dimensional vector space over an algebraically
closed field $K$ of characteristic zero. There are many natural and
interesting subvarieties of $V^{\otimes n}$ defined by properties of
tensors, e.g.~the tensor rank $\le r$ locus and its Zariski
closure. The general study of such varieties is called the \emph{geometry
of tensors} \cite{landsberg}; this is a powerful perspective that has
applications ranging from complexity theory, where the border rank of
matrix multiplication is a topic of much research, to phylogenetic
models and commutative algebra.

In \cite{bdes}, we introduced the notion of \defn{$\GL$-variety};
these are certain infinite dimensional varieties on which the infinite
general linear group $\GL$ acts. These varieties are interesting in
their own right, but are also particularly well-suited to study geometric aspects of tensors. In \cite{bdes}, we established a number of basic properties of $\GL$-varieties; for instance, we proved a version of Chevalley's classical theorem on constructible sets. The purpose of this paper is to prove some finer (and more difficult) theorems about $\GL$-varieties: we construct curves on $\GL$-varieties, and show that the image closure of a map of $\GL$-varieties can be described by limits, just as in classical algebraic geometry. As one might expect, these results have applications to the geometry of tensors.

\subsection{Application to de-bordering}

Before stating our main theorems, which require some technical preliminaries, we describe a concrete application of them to tensor ranks. There are many notions of tensor rank. Typically, the rank $\le r$ locus is not Zariski closed, and one defines a notion of border rank by considering the closure of this set. A \defn{de-bordering} result bounds the border rank above by a function of the rank. De-bordering is central to questions in geometric complexity theory, such as Valiant's conjecture on the permanent or the approach of Mulmuley and Sohoni to P vs.\ NP; see \cite{dgijl}.

We prove a very general de-bordering theorem. We state a version here
for symmetric tensors. Fix $d \ge 1$. Suppose that for each finite
dimensional vector space $V$ we have a subset $Z(V)$ of $\Sym^d(V)$, the ``pure tensors,'' satisfying the following three conditions:
\begin{enumerate}
\item $Z(V)$ contains a non-zero form when $V$ is non-zero;
\item $Z(V)$ is a Zariski closed subset of $\Sym^d(V)$; and 
\item $Z$ is functorial, that is, if $V \to W$ is a linear map then the induced map $\Sym^d(V) \to \Sym^d(W)$ carries $Z(V)$ into $Z(W)$.
\end{enumerate}
Let $x \in \Sym^d(V)$ be a symmetric tensor. We define the
\defn{$Z$-rank} of $x$, denoted $\rank_Z(x)$, to be the minimal $r \in
\bN=\{0,1,2,\ldots\}$ such that $x$ is contained in $rZ(V)=Z(V)+\cdots+Z(V)$; the first and
last condition imply that such an $r$ exists. 
Similarly, we define the \defn{border $Z$-rank} of $x$, denoted $\brank_Z(x)$, to be the minimal $r$ such that $x$ is contained in the Zariski closure of $rZ(V)$. Some examples:
\begin{itemize}
\item Take $Z(V)$ to be the set of $d$th powers, i.e., elements of the form $\ell^d$ for $\ell \in V$. In this case, $Z$-rank is Waring rank.
\item Take $Z(V)$ to be the set of all polynomials that factor as $\ell \cdot f$ where $\ell \in V$ and $f \in \Sym^{d-1}(V)$. In this case, $Z$-rank is \defn{slice rank} \cite{bbov} or {\em q-rank} \cite{des} (for $d=3$). It is known that the slice rank $\le r$ locus is closed.
\item Take $Z(V)$ to be the set of all polynomials of the form $\ell_1
\cdots \ell_d$ with $\ell_i \in V$. This is the Chow variety, and the
$Z$-rank is called the Chow rank. Motivated by a relation to Valiant's
conjecture, equations for the border Chow rank $\leq r$ locus are studied in 
\cite{gu}.
\item Take $Z(V)$ to be the set of all reducible polynomials, i.e., those elements of $\Sym^d(V)$ that factor as $fg$ where $f \in \Sym^e(V)$ and $g \in \Sym^{d-e}(V)$ with $0<e<d$. In this case, $Z$-rank is called \defn{strength} or \defn{Schmidt rank}, and has played an important role in commutative
algebra \cite{ah}, number theory \cite{schmidt}, and algebraic combinatorics \cite{kazi}. It is known that strength and border strength are unequal in general \cite{bbov2}.
\item There are many other possible choices for $Z$. For example, one could take $Z(V)$ to be polynomials that factor as $q \cdot f$ where $q \in \Sym^2(V)$ and $f \in \Sym^{d-2}(V)$. We do not think this has been studied previously  (for $d \ge 2$). This kind of example shows the flexibility of our framework.
\end{itemize}
We prove the following de-bordering theorem:

\begin{theorem} \label{thm:DeBorderingSymd}
Let $Z$ be as above. There is a function $\Phi \colon \bN \to \bN$, depending only on $Z$, such that $\brank_Z(x) \le \Phi(\rank_Z(x))$ for all $x \in \Sym^d(V)$ and all finite dimensional $V$.
\end{theorem}

The interesting part of the theorem is that the bound is independent of the dimension of the ambient vector space $V$. We also prove the following closely related result:

\begin{theorem} \label{thm:UniformLimitSymd}
Let $Z$ be as above. Given $r \ge 0$, there is a constant $\sigma_{Z,r}$, depending only on $Z$ and $r$, with the following property. If $x \in \Sym^d(V)$ has border $Z$-rank $\le r$ then there is an expression
\begin{displaymath}
x = \lim_{t \to 0} \frac{1}{t^s} \sum_{i=1}^r y_i(t),
\end{displaymath}
where $s \le \sigma_{Z,r}$ and $y_i(t)$ is a $K\lbb t \rbb$-point of $Z$.
\end{theorem}

Again, the key point is that the bound $\sigma_{Z,r}$ is independent of the dimension of $V$. This theorem roughly means that one does not require increasingly complicated limits to detect border rank as the dimension of $V$ grows. 

Both of these theorems are well-known and straightforward in the
linear regime, where the dimension of $Z(V)$ grows as a linear function of $\dim(V)$
for $\dim(V) \gg 0$. Indeed, in this case it follows that any element
in $\Sym^d(V)$ of $Z$-rank $\leq r$ lies in $r Z(U) \subseteq \Sym^d(U)$
for some subspace $U$ of $V$ whose dimension is bounded by some constant
independent of $V$, and the results follow by applying standard results
in algebraic geometry to the addition maps $Z(U)^r \to \Sym^d(U)$. This
covers the cases of Waring rank and Chow rank above. The other cases do
not fall in the linear regime (for $d \geq 3$).

In the linear regime, research along the lines of the two theorems
above goes back (at least) to \cite{ll}, where upper bounds on the
corresponding constants $\sigma_{Z,r}$ were found for the ordinary rank
of three-tensors.  Surprising progress on de-bordering Waring rank was
recently made in \cite{dgijl}: there, an upper bound is established on
the Waring rank of a form that is exponential in its border rank but
{\em linear} in $d$.

For strength, the first theorem can also be proved in a different
manner---but this uses a special property of strength, namely, its
universality among all notions of rank. For details, see \cite{bdlz}. 

Beyond the linear regime and the first theorem for the special case
of strength, the two theorems above are entirely new. Unfortunately,
our methods do not give effective bounds for the function $\Phi$ or
the numbers $\sigma_{Z,r}$. We hope that our paper will lead to further
research on deriving such bounds.

\subsection{$\GL$-varieties}

We now turn to stating our main results. For this, we require the notion of $\GL$-variety, which we now recall.
Let $\GL=\bigcup_{n \ge 1} \GL_n(K)$ be the infinite general linear group, and let $\bV=\bigcup_{n \ge 1} K^n$. Let $\bV_{\lambda}$
be the irreducible polynomial representation of $\GL$ corresponding to the partition
$\lambda$, i.e., $\bV_{\lambda}=\bS_{\lambda}(\bV)$, where $\bS_{\lambda}$ is the Schur
functor. Let $\bA^{\lambda}$ be the affine scheme $\Spec(\Sym(\bV_{\lambda}))$. For example,
if $\lambda=(d)$ then $\bA^{\lambda}$ is the space of symmetric $d$-forms on $\bV$. For a tuple of partitions $\ulambda=[\lambda_1, \ldots, \lambda_r]$, we let $\bA^{\ulambda}=\bA^{\lambda_1} \times \cdots \times \bA^{\lambda_r}$.

An affine \defn{$\GL$-variety} is a closed $\GL$-stable reduced subscheme of $\bA^{\ulambda}$. The $\bA^{\ulambda}$ themselves are the basic examples of $\GL$-varieties. 
Any finite dimensional variety can be regarded as a $\GL$-variety with trivial $\GL$ action.
A more interesting example: the border strength $\le r$ locus in $\bA^{(d)}$ is a $\GL$-variety.

\subsection{Existence of curves}

Curves are plentiful in classical algebraic geometry. For
example, any two points on an irreducible variety can be
joined by a curve; this is a fairly easy---yet extremely
important---property. In infinite dimensional algebraic
geometry, curves can be harder to come by. (For us, a
``curve'' must be finite type over a field.) Indeed, it is
not difficult to construct a closed subscheme of
$\bA^{\infty}$ admitting no non-constant maps from any
curve; see Example~\ref{ex:no-curve}. Thus one needs
non-trivial theorems to construct curves on $\GL$-varieties.

An \defn{elementary $\GL$-variety} is one of the form $B
\times \bA^{\ulambda}$, where $B$ is a finite dimensional
variety. Since $\bA^{\ulambda}$ is a linear space, it is
easy to construct curves in it: one can just draw straight
lines. It is therefore easy to construct curves in
elementary $\GL$-varieties (see Lemma~\ref{lem:elem-curve}
for details). One of the main results from \cite{bdes} is
the \defn{decomposition theorem}, which asserts that every
$\GL$-variety admits a finite stratification by
``locally elementary'' varieties. Thus each stratum will contain
plenty of curves.

The decomposition theorem (and indeed, \cite{bdes} in general) gives no information about how the strata are glued, and there could potentially be pathological behavior. For example, suppose $X$ is a $\GL$-variety admitting a locally elementary stratification
\begin{displaymath}
X = 0 \sqcup V \sqcup U,
\end{displaymath}
where 0 is a closed point, $0 \sqcup V$ is the closure of $V$, and $U$ is open. It is conceivable that every curve contained in $U$ could limit on~0, i.e., there are no curves connecting a point in $U$ to a point in $V$. Our first main theorem shows that this does not happen (see Theorem~\ref{thm:curve} for a more precise formulation).

\begin{theorem} \label{mainthm0}
Let $X$ be an irreducible $\GL$-variety, let $U$ be a non-empty open subset of $X$, and let $x$ be a point of $X$. Then there is an irreducible smooth curve $C$ and a map $i \colon C \to X$ such that $\im(i)$ contains $x$ and meets $U$.
\end{theorem}

This theorem provides control on how the strata are glued,
which is really the main achievement of this paper. The
proof of this result is harder (and much more geometric)
than the results of \cite{bdes}. In light of Theorem~\ref{mainthm0}, one might ask if any two points in an irreducible $\GL$-variety can be joined by an irreducible curve. In \cite{bdes2}, we will show that this is indeed the case. The proof builds on the results here, but requires some new elements as well.

\subsection{Image closures} \label{ss:intro-img}

Let $\phi \colon \bA^n \to \bA^m$ be a map of finite dimensional affine spaces. An important
basic result in algebraic geometry states that points in the closure of $\im(\phi)$ can be
described as 1-parameter limits: if $x \in \ol{\im(\phi)}$, then we can write
\begin{equation} \label{eq:limit}
x = \lim_{t \to 0} \phi(t^{-a} y_{-a} + t^{-a+1} y_{-a+1} + \cdots )
\end{equation}
for some integer $a$ and $y_{-a}, y_{-a+1}, \ldots \in
\bA^n$.  The limit notation above really means that the quantity inside
the limit is a polynomial function of $t$, and we are evaluating it
at~0. Our second main theorem establishes an analog of this for $\GL$-varieties:

\begin{theorem} \label{mainthm1}
Let $\phi \colon \bA^{\umu} \to \bA^{\ulambda}$ be a map of $\GL$-varieties. Suppose $x$ is a $K$-point of $\ol{\im(\phi)}$. Then we have
\begin{displaymath}
x = \lim_{t \to 0} \phi(t^{-a} y_{-a} + t^{-a+1} y_{-a+1} + \cdots)
\end{displaymath}
for $K$-points $y_a, y_{a+1}, \ldots$ of $\bA^{\umu}$. Moreover, there exists $N \ge 0$, depending only on $\phi$, such that we can take $0 \le a \le N$ for all $x$.
\end{theorem}

This theorem really contains two separate results. The first is that every point of $\ol{\im(\phi)}$ can be realized as a limit; this is Theorem~\ref{thm:limit}. The second result is the uniformity, i.e., $a$ can be controlled independently of $x$; this is Theorem~\ref{thm:Uniformity}. We deduce our result on border strength (Theorem~\ref{thm:UniformLimitSymd}) from Theorem~\ref{mainthm1}.

It can be useful to recast the above discussion into somewhat more abstract (``valuative'') terms. We explain the basic idea here, because there is one key point to emphasize. First suppose that $\phi \colon \bA^n \to \bA^m$ is a map of finite dimensional affine spaces and $x \in \ol{\im(\phi)}$. The equation \eqref{eq:limit} can be formulated as follows: there exists a $K\lpp t \rpp$-point $y$ of $\bA^n$ such that $x=\lim_{t \to 0} \phi(y(t))$.

Now suppose that $\phi \colon \bA^{\umu} \to \bA^{\ulambda}$ is a map of $\GL$-varieties, and $x \in \ol{\im(\phi)}$. Then Theorem~\ref{mainthm1} shows that there is a $K\lpp t \rpp$-point $y$ of $\bA^{\umu}$ such that $y=\lim_{t \to 0} \phi(y(t))$. However, the point $y$ has an important property: its denominators are bounded. Since $\bA^{\umu}$ has infinitely many coordinates, a general $K\lpp t \rpp$-point can have unbounded denominators. We define a notion of bounded $K\lpp t \rpp$-point for arbitrary $\GL$-varieties, and use this to formulate our general result on image closures.

\subsection{The Krull intersection theorem}

We mention one more application of our results. A \defn{$\GL$-algebra} is an algebra in the category of polynomial representations of $\GL$; see \S \ref{ss:glvar} for details and additional terminology. The basic example of such an algebra is the coordinate ring of a $\GL$-variety. Using Theorem~\ref{mainthm0}, we establish the following version of the Krull intersection theorem in this context:

\begin{theorem} \label{thm:krull}
Let $A$ be a $\GL$-algebra that is reduced and finitely $\GL$-generated, and let $\fp$ be a $\GL$-stable prime ideal of $A$. Then $\bigcap_{n \ge 1} \fp^n A_{\fp}=0$.
\end{theorem}

A major open problem is to show that a finitely $\GL$-generated $\GL$-algebra is noetherian, in the sense that the ascending chain condition holds for $\GL$-stable ideals. The second author \cite{draisma} proved topological noetherianity, i.e., ACC holds for radical $\GL$-stable ideals; this is a powerful result, but much weaker than true noetherianity. Theorem~\ref{thm:krull} could potentially be an important tool in establishing noetherianity. In fact, Theorem~\ref{thm:krull} has a more immediate application too: it plays an important role in \cite{DS}, which studies the singular locus of a $\GL$-variety.

\subsection{Overview of the proofs}

We now outline the proofs of our main theorems. In what follows, $X$ is an irreducible affine $\GL$-variety.

\textit{The shift theorem for pairs.} A fundamental result
from \cite{bdes} is the shift theorem. This asserts that $X$
contains a non-empty open subvariety $U$, preserved by a
shifted $\GL$-action that makes $U$ an elementary
$\GL$-variety. This is a very useful result since it provides an open subvariety with a very simple structure. Unfortunately, this result provides no information on $U$. In particular, if $Z$ is a given proper closed subvariety, one cannot ensure $Z \cap U$ is non-empty.

The first important result we prove here is a shift theorem for pairs (Theorem~\ref{thm:cdone}): given a closed $\GL$-subvariety $Z \subset X$ of codimension one, we can find $U$ as above such that $U \cap Z$ is non-empty and locally elementary. To prove this, we first construct a map of $\GL$-varieties $X \to Y$ such that $Z \to Y$ is dominant of relative dimension~0. We then construct $U$ simply by shrinking $Y$, which ensures that we do not discard all of $Z$. We use an important finiteness property of curves in this argument
(Theorem~\ref{thm:subalg}), which is why it is important that $Z$ has codimension one.

\textit{Unirationality for pairs.} Another important result from \cite{bdes} is the unirationality theorem. This asserts that if $X$ is an irreducible $\GL$-variety then there is a dominant morphism $\phi \colon B \times \bA^{\ulambda} \to X$ for some finite dimensional irreducible variety $B$ and tuple $\ulambda$. As with the shift theorem, this theorem provides no information on $\im(\phi)$, e.g., one cannot ensure $\im(\phi)$ meets a given closed subset of $X$.

The second important result we prove here is a unirationality theorem for pairs (Theorem~\ref{thm:uni-pair}). Roughly, this says that if $Z$ is a given closed $\GL$-subvariety of $X$ then one can find $\phi$ as in the previous paragraph that restricts to a similar such map over $Z$; in particular, $\im(\phi)$ contains a dense open subset of $Z$. When $Z$ has codimension one, this follows immediately from the shift theorem for pairs. In general, we blow-up along $Z$ to reduce to the codimension one case. To carry this out, we develop basic aspects of blow-ups for $\GL$-varieties, without aiming for generality.

The unirationality theorem for pairs can be viewed as the core technical result in this paper: it implies that the closed $\GL$-subvariety $Z$ is not glued to its complement in a pathological way, and this is what allows us to get control over the stratification in the decomposition theorem.

\textit{Proof sketch of Theorem~\ref{mainthm0}.} Let $x \in X$ and $U \subset X$ be given as in the theorem. Let $\phi \colon B \times \bA^{\ulambda} \to X$ be a dominant map with $x \in \im(\phi)$; this can be constructed using the unirationality theorem for pairs by taking $Z$ to be the orbit closure of $x$. By Chevalley's theorem for $\GL$-varieties (proved in \cite{bdes}), $\im(\phi)$ contains a dense open subset of $X$, and thus meets $U$; let $y$ belong to $\im(\phi) \cap U$. Now simply choose pre-images of $x$ and $y$, connect these by a curve in $B \times \bA^{\ulambda}$ (which is easy), and map this curve to $X$ via $\phi$.

\textit{Proof sketch of Theorem~\ref{mainthm1}.} Let $X \subset \bA^{\ulambda}$ be the image closure of $\phi$, and let $x \in X$ be given. Let $U$ be an open subset of $X$ contained in $\im(\phi)$, which exists by Chevalley's theorem for $\GL$-varieties. Find a curve joining $x$ to $U$. Replacing this curve with its completion at a point mapping to $x$, the generic point is a $K\lpp t \rpp$-point $\gamma(t)$ of $U$ such that $x=\lim_{t \to 0} \gamma(t)$. Since $U \subset \im(\phi)$, we can lift $\gamma$ to a $K\lpp t^{1/n} \rpp$-point $\tilde{\gamma}(t)$ of $\bA^{\umu}$ with bounded denominators; this is not obvious, but follows from from the decomposition theorem. We thus have $x=\lim_{t \to 0} \phi(\tilde{\gamma}(t))$, which realizes $x$ as a limit.

We now explain the uniformity aspect. The above argument implies that every point of $X$ can be realized as a limit along a curve in $\bA^{\umu}$. Consider a family of curves in $\bA^{\umu}$ over a finite dimensional base variety. Using Chevalley's theorem for $\GL$-varieties, we show that the locus in $X$ that can be realized as a limit using a curve from this family is a constructible set. Since $X$ is the union of these constructible sets, it is the union of some finite collection of them, and this yields uniformity. (This discussion is greatly oversimplified; see \S \ref{s:app} for details.)

\subsection{Connections to other work}

Theorem~\ref{mainthm1} is used in \cite{bdv}, in which a
theoretical algorithm is derived for computing the image closure of
a morphism between $\GL$-varieties. In that algorithm, an important
subroutine decides whether, for two morphisms $\alpha,\alpha'$
into the the same $\GL$-variety, $\im(\alpha')$ is contained in
$\overline{\im(\alpha)}$. A no-instance of this subroutine can be
certified with an equation for $\overline{\im(\alpha)}$ that does
not vanish on $\im(\alpha')$, but a certificate for a yes-instance is
much more subtle and involves (a computer representation of) 
limits as in 
Theorem~\ref{mainthm1}. Theorem~\ref{mainthm1} is also an important input
to \cite{bdes2}, which proves an improved form of the unirationality
theorem.

\subsection{Notation}

We typically denote ordinary, finite dimensional varieties by $B$ or $C$, and $\GL$-varieties by $X$ or $Y$. Other important notation:
\begin{description}[align=right,labelwidth=2cm,leftmargin=!]
\item [$K$] the base field, of characteristic~0
\item [$\Omega$] an extension of $K$
\item [$\GL$] the infinite general linear group
\item [$\ulambda$] a tuple of partitions $[\lambda_1, \ldots, \lambda_r]$
\item [$\bA^{\ulambda}$] the basic affine $\GL$-variety
\item [{$K[X]$}] the coordinate ring of the affine scheme $X/K$
\end{description}
A \defn{variety} over a field $\Omega$ is a reduced scheme that is separated and of finite type over $\Omega$. Most varieties in this paper are affine. A \defn{curve} is a
one-dimensional variety; we do not require curves to be smooth or irreducible, unless explicitly stated.

\subsection*{Acknowledgements}

We thank Christopher Heng Chiu for helpful conversations, and for providing a reference for Theorem~\ref{thm:subalg}. 

\section{Background on \texorpdfstring{$\GL$}{GL}-varieties} \label{s:bg}

\subsection{\texorpdfstring{$\GL$}{GL}-varieties} \label{ss:glvar}

Fix a field $K$ of characteristic~0. Let $\GL=\bigcup_{n \ge 1}
\GL_n(K)$ be the infinite general linear group, and let
$\bV=\bigcup_{n \ge 1} K^n$ be its standard representation. For a
partition $\lambda$, we let $\bV_{\lambda}=\bS_{\lambda}(\bV)$ be the
irreducible polynomial representation of $\GL$ with highest weight
$\lambda$; here $\bS_{\lambda}$ denotes a Schur functor. For example,
$\bV_{(d)}$ is the $d$th symmetric power of $\bV$. We let
$\bA^{\lambda}$ be the spectrum of the polynomial ring
$\Sym(\bV_{\lambda})$; this is an affine scheme equipped with an
action of the group $\GL$. For a tuple of partitions
$\ulambda=[\lambda_1, \ldots, \lambda_r]$, we write
$\bV_{\ulambda}=\bV_{\lambda_1}\oplus\cdots\oplus\bV_{\lambda_r}$ and let
$\bA^{\ulambda}=\bA^{\lambda_1} \times \cdots \times \bA^{\lambda_r}$ be the spectrum of $\Sym(\bV_{\ulambda})$.
The role of the spaces $\bA^{\ulambda}$ in our theory is analogous to
the role of the usual affine spaces $\bA^n$ in classical algebraic
geometry.

A \defn{$\GL$-algebra} is a $K$-algebra equipped with an action of $\GL$
by automorphisms that makes it into a polynomial $\GL$-representation.
We say that it is \defn{finitely $\GL$-generated} if it is generated
by the $\GL$-orbits of finitely many elements; equivalently, if it
is isomorphic to a quotient of some $\Sym(\bV_{\ulambda})$ by some
$\GL$-stable ideal. An \defn{affine $\GL$-scheme} is the spectrum of
a $\GL$-algebra.  An \defn{affine $\GL$-variety} is a reduced affine
$\GL$-scheme whose coordinate ring is finitely $\GL$-generated;
equivalently, it is a $\GL$-scheme isomorphic to a reduced closed
subscheme of some $\bA^{\ulambda}$. A \defn{quasi-affine $\GL$-variety} (or
\defn{$\GL$-variety}, for short), is a $\GL$-stable open subscheme of an affine
$\GL$-variety. A \defn{morphism of $\GL$-varieties} is a $\GL$-equivariant morphism of
varieties. 

Every $\GL$-algebra $R$ has a natural grading $R=\bigoplus_{d \geq 0}
R_d$, in which $R_d$ consists of all $f \in R$ with the property that,
for all $n \gg 0$, the element $g:=t\cdot \id \in \GL_n(K)$ satisfies
$gf=t^d f$. This grading is called the {\em central grading} of $R$.
For example, if $R = \Sym(\bV_{\lambda})$ for a partition $\lambda$,
then the elements of $\bV_{\lambda}$ have degree
$|\lambda|$, the number of which $\lambda$ is a partition.

We note that the central grading of $R$ gives rise to an action of $K^\times$ on $R$
and on $\Spec(R)$ that does not factor through a group homomorphism
$K^\times \to \GL$: $t$ times the infinite identity matrix is not an
element of $\GL$.

\subsection{$\GL$-varieties as functors} \label{ss:functorial}

For an integer $n\geq0$, the inclusion $K^n\to\bV$ yields a surjective morphism 
\[
\bA^{\ulambda}\to\Spec(\Sym(\bS_{\lambda_1}(K^n)\oplus\cdots\oplus\bS_{\lambda_r}(K^n)))=:\bA^{\ulambda}\{K^n\},
\]
which has a section corresponding to the natural surjection
$\bV \to K^n$.

If $X$ is a closed $\GL$-subvariety of $\bA^{\ulambda}$, then we
write $X\{K^n\}$ for the image of $X$ in $\bA^{\ulambda}\{K^n\}$;
this is a $\GL_n(K)$-stable closed subvariety, and the section
$\bA^{\ulambda}\{K^n\} \to \bA^{\ulambda}$ restricts to a
$\GL_n(K)$-equivariant morphism from $X\{K^n\}$ into $X$.

More generally, by exploiting the Schur functors in the definition
of $\GL$-varieties, one can regard an affine $\GL$-variety $X$ as a
functor from the category of finite-dimensional vector spaces over $K$
to the category of affine varieties over $K$; this is the viewpoint
taken in \cite{draisma}. Accordingly, we write $X\{V\}$ when $V$ is a
finite-dimensional vector space over $K$, and similarly for morphisms.
One difference with \cite{draisma} is that since $\bS_{\lambda}$ is a
covariant functor, $\bA^{\ulambda}$ are naturally contravariant in $X$.

\subsection{Shifting}

For an integer $n\geq0$, let $G(n)$ be the subgroup of $\GL$ consisting of block matrices of the form
\begin{displaymath}
\begin{pmatrix} \id_n & 0 \\ 0 & \ast \end{pmatrix}.
\end{displaymath}
The group $G(n)$ is in fact isomorphic to $\GL$. Thus if $X$ is a $\GL$-variety, we can restrict the action of $\GL$ to $G(n)$ and then identify $G(n)$ with $\GL$ to obtain a new action of $\GL$. We call this the $n$th \defn{shift} of $X$, and denote it by $\Sh_n(X)$. 

\begin{theorem}[{\cite[Theorem 5.1]{bdes}}]\label{thm:shift}
Let $X$ be a non-empty affine $\GL$-variety. Then there is a non-empty open affine $\GL$-subvariety of $\Sh_n(X)$, for some $n$, that is isomorphic to $B \times \bA^{\ulambda}$ for some variety $B$ and tuple $\ulambda$.
\end{theorem}

We call a tuple $\ulambda$ pure when it does not contain the empty partition. Note that every
tuple $\ulambda$ is the concatenation of a tuple of $m\geq0$ empty partitions and a pure
tuple $\umu$. We have $\bA^{\ulambda}=\bA^m\times\bA^{\umu}$. By replacing $B$ 
by $B \times \bA^m$, we can always assume the tuple $\ulambda$ in the above theorem to be pure.

\subsection{Decomposition theorem}

We say that a $\GL$-variety $X$ is \defn{elementary} if it is isomorphic to a $\GL$-variety of the form $B \times \bA^{\ulambda}$ for some irreducible affine variety $B$ and tuple $\ulambda$; this implies that $X$ is irreducible and affine. Let $\phi \colon X \to Y$ be a morphism of $\GL$-varieties. We say that $\phi$ is \defn{elementary} if there exists a commutative diagram
\begin{equation} \label{eq:Diagram}
\begin{gathered}
\xymatrix@C=5em{
X \ar[r]^{\phi} \ar[d]_-i & Y \ar[d]^-j \\
B \times \bA^{\ulambda} \ar[r]^{\psi \times \pi} & C \times \bA^{\umu} }
\end{gathered}
\end{equation}
where $i$ and $j$ are isomorphisms of $\GL$-varieties, $B$ and $C$ are irreducible affine varieties, $\psi \colon B \to C$ is a surjective morphism of varieties, $\ulambda$ and $\umu$ are pure tuples with $\umu \subseteq \ulambda$, and $\pi \colon \bA^{\ulambda} \to \bA^{\umu}$ is the projection map. This implies that $X$ and $Y$ are elementary and that $\phi$ is surjective. Note that a $\GL$-variety is elementary if and only if its identity map is. We say that a $G(n)$-variety, or morphism of $G(n)$-varieties, is \defn{$G(n)$-elementary} if it is so after identifying $G(n)$ with $\GL$.

We say that a quasi-affine $\GL$-variety $X$ is \defn{locally elementary at $x \in X$} if there exists $n$ and a $G(n)$-stable open neighborhood of $x$ that is $G(n)$-elementary; we say that $X$ is \defn{locally elementary} if it is so at all points. Let $\phi \colon X \to Y$ be a morphism of $\GL$-varieties. We say that $\phi$ is \defn{locally elementary at $x \in X$} if there exists $n$ and $G(n)$-stable open neighborhoods $U$ of $x$ and $V$ of $\phi(x)$ such that $\phi$ induces a $G(n)$-elementary morphism $U \to V$. We say that $\phi$ is \defn{locally elementary} if it is surjective and locally elementary at all $x \in X$. This implies that both $X$ and $Y$ are locally elementary. We note that a $\GL$-variety is locally elementary if and only if its identity morphism is. We again define a $G(n)$-variety, or a morphism of $G(n)$-varieties, to be \defn{locally $G(n)$-elementary} if it is so after identifying $G(n)$ with $\GL$.

\begin{theorem}[{\cite[Theorem 7.8]{bdes}}]\label{thm:decomp}
Every morphism of $\GL$-varieties admits a locally elementary decomposition.
\end{theorem}

\subsection{Chevalley's theorem}
Theorem~\ref{thm:decomp} allows one to easily prove a variety of results about $\GL$-varieties, such as:

\begin{theorem}[{\cite[Theorem 7.13]{bdes}}]\label{thm:chev}
Let $\phi \colon X \to Y$ be a morphism of quasi-affine $\GL$-varieties and let $C$ be a $\GL$-constructible subset of $X$. Then $\phi(C)$ is a $\GL$-constructible subset of $Y$.
\end{theorem}

Here, we say that a subset of a quasi-affine $\GL$-variety is \defn{$\GL$-constructible} if it is a finite union of locally closed $\GL$-stable subsets. One consequence of this theorem is that if $\phi$ is surjective on $\ol{K}$-points then it is scheme-theoretically surjective.

\section{The shift theorem for pairs} \label{s:shift}

\subsection{Statement of results}

In \cite{bdes}, we proved a number of structural results for $\GL$-varieties. An interesting (and seemingly difficult) problem is to extend these results to relative situations, such as a morphism of $\GL$-varieties or just a $\GL$-variety together with a subvariety. The purpose of this section is to prove a relative form of the shift theorem:

\begin{theorem} \label{thm:cdone}
Let $X$ be an irreducible affine $\GL$-variety and let $Z$ be an irreducible closed $\GL$-subvariety of codimension $1$. Then there exists a commutative diagram
\begin{displaymath}
\xymatrix@C=4em{
B \times \bA^{\ulambda} \ar[r]^{j \times \id} \ar@{=}[d] &
C \times \bA^{\ulambda} \ar@{=}[d] \\
\Sh_n(Z)[1/h] \ar[r] & \Sh_n(X)[1/h] }
\end{displaymath}
where $n \in \bN$, $h$ is an invariant function on $\Sh_n(X)$ that is not identically zero on $\Sh_n(Z)$, $\ulambda$ is a pure tuple, $C$ is an irreducible affine variety, $B$ is an irreducible closed subvariety of $C$ of codimension one, and $j \colon B \to C$ is the inclusion.
\end{theorem}

As an immediate corollary, we obtain a version of the unirationality theorem for pairs:

\begin{corollary} \label{cor:shift-pair}
Let $Z \subset X$ be as in the theorem. Then there is a commutative diagram
\begin{displaymath}
\xymatrix@C=4em{
B \times \bA^{\ulambda} \ar[r]^{j \times \id} \ar[d]_{\psi} &
C \times \bA^{\ulambda} \ar[d]^{\phi} \\
Z \ar[r] & X }
\end{displaymath}
where $j \colon B \to C$ is as in the theorem, and $\phi$ and $\psi$ are dominant morphisms of $\GL$-varieties.
\end{corollary}

\begin{proof}
There is a natural map $\Sh_n(X)[1/h] \to X$ that is dominant, see
\cite[Proposition 4.4]{bdes}, and so is its restriction to $\Sh_n(Z)[1/h] \to Z$.
\end{proof}

We now say a bit about the proof of Theorem~\ref{thm:cdone}, which will take all of \S \ref{s:shift}. The shift theorem proved in \cite{bdes} is for a single $\GL$-variety: it says that we can find $n$ and $h$ so that $\Sh_n(X)[1/h]$ is elementary. The proof gives essentially no information about $h$. In particular, we cannot necessarily choose $h$ to be non-zero on a given $\GL$-subvariety $Z \subset X$. This is the main difficulty we must overcome to prove Theorem~\ref{thm:cdone}.

The proof of the theorem has two main components. The first (Proposition~\ref{prop:sub-dim-zero}) shows that (after shifting and shrinking) we can find a map $\phi \colon X \to Y$ such that $\phi \vert_Z$ is dominant of relative dimension zero. Our strategy is to then prove a shift theorem for $\phi$ that only involves shrinking $Y$; this will ensure that we do not discard $Z$ in the shrinking process. We have not been able to prove such a shift theorem for general maps. However, since $Z$ has codimension one, the map $\phi$ has relative dimension one. The second main component (Proposition~\ref{prop:rel-dim-one}) is a shift theorem for such morphisms. The proof leverages an important finiteness property of curves (Theorem~\ref{thm:subalg}), which is why the relative dimension one condition is important.

To facilitate the above arguments, we first develop some general theory in \S \ref{ss:reldim}--\S \ref{ss:spread}.

\subsection{Relative dimension} \label{ss:reldim}

Let $X$ be a $\GL$-variety. For $m \in \bN$, we let $\delta_X(m)=\dim X\{K^m\}$. We refer to $\delta_X$ as the \defn{dimension function} of $X$. We showed \cite[Proposition 5.10]{bdes} that $\delta_X$ is eventually polynomial, i.e., there exists a polynomial $p \in \bQ[t]$ such that $\delta_X(m)=p(m)$ for all $m \gg 0$. Let $Y$ be a second $\GL$-variety. For an integer $d$, we write $\dim(X|Y)=d$ if $\delta_X(m)-\delta_Y(m)=d$ for all $m \gg 0$. We think of $\dim(X|Y)$ as the relative dimension between $X$ and $Y$. We note that if $\dim(X|Y)=d$, then also $\dim(\Sh_n(X)|\Sh_n(Y))=d$ for any $n$. If $Y$ is a subvariety of $X$ then we refer to $\dim(X|Y)$ as the codimension of $Y$ in $X$.  If $\phi \colon X \to Y$ is a map of $\GL$-varieties, we refer to $\dim(X|Y)$ as the relative dimension of $\phi$.

\subsection{Invariant finite presentations}

We say that a morphism $S \to R$ of $\GL$-algebras is of \defn{invariant finite presentation} if $R$ is isomorphic to an $S$-algebra of the form 
\[
S[x_1, \ldots, x_n]/(f_1, \ldots, f_m)
\]
where the $x$'s and $f$'s are $\GL$-invariant. We say that a morphism $X \to Y$ of affine $\GL$-varieties is of \defn{invariant finite presentation} if the map on coordinate rings is. We note that if $X \to Y$ is of invariant finite presentation then so is $\Sh_n(X) \to \Sh_n(Y)$ for any $n$, as well as $X[1/h] \to Y[1/h]$ for any invariant function $h$ on $Y$.

The following gives a canonical characterization of invariant finite presentation; in it,
$R_0$ stands for the part of a $\GL$-algebra $R$ of degree zero in the central grading. 

\begin{proposition} \label{prop:ifp-char}
Let $f \colon S \to R$ be a map of $\GL$-algebras. The following are equivalent:
\begin{enumerate}
\item The map $f$ is of invariant finite presentation.
\item The map $S_0 \to R_0$ is of finite presentation and the natural map $R_0 \otimes_{S_0} S \to R$ is an isomorphism.
\end{enumerate}
\end{proposition}

\begin{proof}
Suppose (a) holds. Write $R=S[x_1, \ldots, x_n]/(f_1, \ldots, f_m)$ with the $x$'s and $f$'s $\GL$-invariant. Looking in degree~0, we see that $R_0=S_0[x_1, \ldots, x_n]/(f_1, \ldots, f_m)$. It is therefore clear that (b) holds.

Now suppose (b) holds. We thus have an isomorphism $R_0=S_0[x_1, \ldots, x_n]/(f_1, \ldots, f_m)$. Applying $- \otimes_{S_0} S$, we find
\begin{displaymath}
R = R_0 \otimes_{S_0} S = S[x_1, \ldots, x_n]/(f_1, \ldots, f_m)
\end{displaymath}
and so (a) holds.
\end{proof}

As a consequence of this proposition, we obtain the following useful transitivity.

\begin{proposition} \label{prop:ifp-trans}
Let $T \to S \to R$ be maps of $\GL$-algebras, with $T \to S$ and $T \to R$ of invariant finite presentation and $S_0$ noetherian. Then $S \to R$ is of invariant finite presentation.
\end{proposition}

\begin{proof}
Since $R$ is finitely generated as a $T$-algebra, it is certainly also finitely generated as an $S$-algebra. Thus $R_0$ is finitely generated over $S_0$, and therefore finitely presented over $S_0$, since $S_0$ is noetherian. Now consider the following diagram
\begin{displaymath}
\xymatrix@C=3em{
R_0 \otimes_{T_0} T \ar@{=}[r] &
R_0 \otimes_{S_0} (S_0 \otimes_{T_0} T) \ar[r]^-i &
R_0 \otimes_{S_0} S \ar[r]^-j & R }
\end{displaymath}
and let $k$ be the composition. Since $T \to S$ and $T \to R$ are of invariant finite presentation, we see that $k$ and $i$ are isomorphisms (Proposition~\ref{prop:ifp-char}). Thus $j$ is an isomorphism. Since $S_0$ is noetherian, it follows that $S \to R$ is of invariant finite presentation (Proposition~\ref{prop:ifp-char}).
\end{proof}

Suppose that $X \to Y$ is a map of affine $\GL$-varieties. Let $X_0$ be the spectrum of the degree~0 piece of $K[X]$, and similarly for $Y_0$. We have a commutative diagram
\begin{displaymath}
\xymatrix{
X \ar[d] \ar[r] & Y \ar[d] \\
X_0 \ar[r] & Y_0 }
\end{displaymath}
Geometrically, Proposition~\ref{prop:ifp-char} says that $X \to Y$ is of invariant finite
presentation if and only if this diagram is cartesian. (We note that $X_0 \to Y_0$ is of
finite presentation, being a map of affine varieties of finite type over $K$.) The following proposition, which is one of the main reasons we care about invariant finite presentation, is an immediate consequence of this:

\begin{proposition} \label{prop:ifp-over-elem}
Let $\phi \colon X \to B \times \bA^{\ulambda}$ be a morphism of affine $\GL$-varieties of
invariant finite presentation, where $B$ is a variety and $\ulambda$ is a pure tuple. Then
$X$ is isomorphic to $C \times \bA^{\ulambda}$ for some variety $C$, and under this
isomorphism $\phi$ corresponds to $\psi \times \id \colon C \times \bA^{\ulambda} \to B
\times \bA^{\ulambda}$ for some morphism $\psi \colon C \to B$. 
\end{proposition}

A map of $\GL$-algebras of finite type certainly does not have to
be of invariant finite presentation: for instance, for any partition
$\lambda$ with $|\lambda| > 0$, the homomorphism $\Sym(\bV_{\lambda})
\to K$ with kernel generated by $\bV_{\lambda}$ is not.  However, the
next proposition does give something in this direction.

\begin{proposition} \label{prop:gen-ifp}
Let $\phi \colon X \to Y$ be a morphism of affine $\GL$-varieties and assume that $\phi$ is of finite type in the usual sense. Then there exists $n \ge 0$ and a nonzero invariant function $h$ on $\Sh_n(Y)$ such that $\Sh_n(X)[1/h] \to \Sh_n(Y)[1/h]$ is flat of invariant finite presentation.
\end{proposition}
\begin{proof}
We are free to shift and shrink $Y$ throughout the proof. By generic flatness \stacks{052B}
there is a non-empty open subset $U$ of $Y$ such that $X_U \to U$ is flat of finite
presentation. Replacing $U$ with $\bigcup_{g \in \GL} gU$, we can assume that $U$ is
$\GL$-invariant. Let $h$ be a nonzero function on $Y$ that vanishes identically on the complement of $U$. Now shift so that $h$ becomes an invariant function and replace $Y$ by $Y[1/h]$. Now $X \to Y$ is flat of finite presentation. Write down a finite presentation for $K[X]$ as a $K[Y]$-algebra and shift so that the generators and relations become invariant. The result follows.
\end{proof}

We warn the reader that, in the above proposition, if $X \to Y$ is not dominant then $\Sh_n(X)[1/h]$ may be the empty variety. Combining the above two results, we obtain the following shift theorem for morphisms of finite type:

\begin{proposition} \label{prop:ft-shift}
Let $\phi\colon X \to Y$ be a dominant morphism of affine $\GL$-varieties such that $\phi$ is of finite type. Then there exists a diagram
\begin{displaymath}
\xymatrix@C=4em{
\Sh_n(X)[1/h] \ar@{=}[r] \ar[d]_{\phi} & C \times \bA^{\ulambda} \ar[d]^{\psi \times \id} \\
\Sh_n(Y)[1/h] \ar@{=}[r] & B \times \bA^{\ulambda} }
\end{displaymath}
where $n$ is a non-negative integer, $h$ is a nonzero invariant function on $\Sh_n(Y)$, $B$ and $C$ are irreducible affine varieties, $\ulambda$ is a pure tuple, and $\psi \colon C \to B$ is a flat dominant map.
\end{proposition}

\subsection{Spreading out finiteness} \label{ss:spread}

In the previous subsection, we established a shift theorem
for a finite-type map $X \to Y$ of $\GL$-varieties. We now show that it is in fact enough to assume that $X \to Y$ is generically of finite type (i.e., the generic fiber is of finite type). This stronger form is used in the proof of Proposition~\ref{prop:rel-dim-one} below. Here is the precise statement:

\begin{proposition} \label{prop:ft-spread}
Let $\phi \colon X \to Y$ be a dominant morphism of irreducible affine $\GL$-varieties. Let $\eta \in Y$ be the generic point, and suppose that $X_{\eta}$ is of finite type over $K(Y)$. Then there exists $n \ge 0$ and a non-zero $\GL$-invariant function $h$ on $Y$ such that $\Sh_n(X)[1/h] \to \Sh_n(Y)[1/h]$ is of finite type.
\end{proposition}

We will require some preparation before giving the proof. Let $V$ be an arbitrary
representation of $\GL$ over some field extension $\Omega$ of $K$. We say that an element of $V$ is \defn{polynomial} if it generates a polynomial subrepresentation. We let $V^{\pol}$ be the set of all polynomial elements; this is the maximal polynomial subrepresentation of $V$. Note that $(-)^{\pol}$ commutes with direct sums (even infinite direct sums). We will need the following lemma. This lemma also appears as \cite[Proposition~2.8]{sym2noeth}, but we include a proof to be self-contained; we note that our proof differs from the one in \cite{sym2noeth}.

\begin{lemma} \label{lem:ft-spread}
Let $\ulambda$ be a pure tuple and let $S=\Omega \otimes_K \Sym(\bV_{\ulambda})$, and denote
by $\Frac(S)$ the fraction field of the domain $S$. Then $\Frac(S)^{\pol}=S$.
\end{lemma}

\begin{proof}
Relabeling, we assume $\Omega=K$ for notational simplicity. Note that $S$ is a polynomial ring over $K$. Let $x$ be an element of $\Frac(S)^{\pol}$, and write $x=f/h$ with $f,h \in S$ coprime polynomials. Let $g \in \GL$. Since $x$ belongs to a polynomial representation, $gx$ is contained in the span of elements of the form $D_1 D_2 \cdots D_m x$, where the $D_i$ are elements of the Lie algebra of $\GL$. The $D_i$'s act by derivations on the ring $S$, and so, by the Leibniz rule, the denominator of $D_1 \cdots D_m x$ is a power of $h$. We thus see that the denominator of $gx$ is a power of $h$ for all $g \in \GL$. This implies that $h$ is $\GL$-invariant, and therefore a scalar, and so $x \in S$ as required.
\end{proof}

\begin{proof}[Proof of Proposition~\ref{prop:ft-spread}]
By Theorem~\ref{thm:shift}, after shifting and shrinking $Y$, we can assume $Y$ is
elementary, i.e., $Y=B \times \bA^{\ulambda}$ with $B$ an irreducible variety and $\ulambda$
a pure tuple. Let $S=K[Y]$ and $R=K[X]$, so that $\phi$ corresponds to an injective
homomorphism $S \to R$. We have $S=K[B] \otimes Q$, where $Q=\Sym(\bV_{\ulambda})$. Let
$L=\Frac(S)$. Then, by assumption, $R_L=R \otimes_{S} L$ is a finitely generated $L$-algebra.
Let $x_1, \ldots, x_n \in R$ generate $R_L$ as an $L$-algebra. Shift further so that each $x_i$ is $\GL$-invariant, and let $R' \subset R$ be the $S$-subalgebra generated by the $x_i$'s.

Since the monomials in the $x_i$'s span $R_L$ as an $L$-vector space, we can choose some subset $E$ of them that form a basis. We thus have $R_L = \bigoplus_{e \in E} Le$. Since the $e$'s are $\GL$-invariant, this is a direct sum decomposition of $\GL$-representations over the coefficient field $K(B)$. We thus have $R_L^{\pol} = \bigoplus_{e \in E} L^{\pol} e$. By Lemma~\ref{lem:ft-spread}, we have $L^{\pol}=K(B) \otimes Q$. We thus see that $R_L^{\pol} = K(B) \otimes_{K[B]} R'$.

Let $y_1, \ldots, y_m$ be elements of $R$ whose $\GL$-orbits generate $R$ as a $K$-algebra. Since $y_i$ is an element of $R \subset R_L^{\pol}$, it belongs to $K(B) \otimes_{K[B]} R'$. Since there are finitely many $y$'s, there is a non-zero element $h \in K[B]$ such that each $y_i$ belongs to $R'[1/h]$. Since $R'[1/h]$ is $\GL$-stable, we see that $gy_i \in R'[1/h]$ for $g \in \GL$. Thus $R[1/h]=R'[1/h]$, and so $R[1/h]$ is generated over $S[1/h]$ by the $x_i$'s, which completes the proof.
\end{proof}

\subsection{Maps of relative dimension one}

We now prove the first key result needed for the proof of Theorem~\ref{thm:cdone}: a shift theorem for maps of relative dimension one.

\begin{proposition} \label{prop:rel-dim-one}
Let $\phi \colon X \to Y$ be a dominant map of irreducible affine $\GL$-varieties such that $\dim(X|Y)=1$. We can then find $n \ge 0$ and a $\GL$-invariant function $h$ on $Y$ such that $\Sh_n(Y)[1/h]$ is elementary and $\Sh_n(X)[1/h] \to \Sh_n(Y)[1/h]$ is of invariant finite presentation.
\end{proposition}

The proposition yields the following more explicit result:

\begin{corollary}
Keeping the notation from the proposition, there is a commutative diagram
\begin{displaymath}
\xymatrix@C=4em{
\Sh_n(X)[1/h] \ar@{=}[r] \ar[d]_{\phi} & C \times \bA^{\ulambda} \ar[d]^{p \times \id} \\
\Sh_n(Y)[1/h] \ar@{=}[r] & B \times \bA^{\ulambda} }
\end{displaymath}
where $\ulambda$ is a pure tuple, $B$ and $C$ are irreducible affine varieties, and $p$ is a dominant map of varieties.
\end{corollary}

\begin{proof}
This follows from the proposition and Proposition~\ref{prop:ifp-over-elem}.
\end{proof}

To prove the proposition, we will require the following theorem. A proof can be found in \cite[Theorem A, part
(2)]{wadsworth}; the author mentions that the result was likely already well-known. We thank Christopher Heng Chiu for pointing out this reference to us.

\begin{theorem} \label{thm:subalg}
Let $\Omega$ be a field and let $A$ be an $\Omega$-algebra that is integral, finitely generated and one-dimensional. Then any $\Omega$-subalgebra of $A$ is finitely generated.
\end{theorem}

This theorem is applied in the following key lemma.

\begin{lemma} \label{lem:rel-dim-one}
Let $\phi \colon X \to Y$ be as in Proposition~\ref{prop:rel-dim-one}, and let $\eta$ be the generic point of $Y$. Then $X_{\eta}$ is of finite type over $K(Y)$.
\end{lemma}

\begin{proof}
We are free to shift and to shrink $Y$. We can thus assume by Theorem~\ref{thm:shift} that $Y=B \times \bA^{\ulambda}$ for some irreducible variety $B$ and pure tuple~$\ulambda$. By further shifting, we can also assume that there is a nonzero invariant function $h$ on $X$, irreducible variety $C$, pure tuple $\umu$ and isomorphism $\gamma\colon C \times \bA^{\umu}\to X[1/h]$. Since $\phi\circ\gamma\colon C \times \bA^{\umu}\to B \times \bA^{\ulambda}$ is dominant, we have $\ulambda\subseteq\umu$ by \cite[Proposition 6.1]{bdes}. From the fact $\dim(X|Y)=1$, we conclude that $\ulambda=\umu$ and $\dim(C)=1+\dim(B)$. Moreover \cite[Proposition 6.1]{bdes} shows that, after shrinking some more and composing with an automorphism of $\bA^{\ulambda}$, we can assume that $\phi\circ\gamma$ has the form $\psi \times \id$ where $\psi \colon C \to B$ is a dominant morphism; note that $C$ is generically a curve over $B$. Now, $X_{\eta}$ is easily seen to be an integral scheme over $K(Y)$. By the above, $X_{\eta}[1/h]$ is a curve over $K(Y)$, that is, it has finite type and dimension~1. By Theorem~\ref{thm:subalg}, we conclude that $X_{\eta}$ is also a curve over $K(Y)$, which completes the proof.
\end{proof}

\begin{proof}[Proof of Proposition~\ref{prop:rel-dim-one}]
We are free to shift $\phi$ and shrink $Y$. By Lemma~\ref{lem:rel-dim-one}, the generic fiber of $X$ is of finite type over $K(Y)$. After shifting and shrinking $Y$, Proposition~\ref{prop:ft-spread} shows that $X \to Y$ is of finite type. The result now follows from Proposition~ \ref{prop:ft-shift}.
\end{proof}

\begin{remark}
Theorem~\ref{thm:subalg} is false in higher dimension, e.g., $K[x,y]$ contains subalgebras that are not finitely generated. We expect Proposition~\ref{prop:rel-dim-one} holds in higher finite relative dimension, but we are only able to prove it in relative dimension one since we make use of Theorem~\ref{thm:subalg}.
\end{remark}

\subsection{Relativizing subvarieties}

We now prove the second key result needed for our proof of Theorem~\ref{thm:cdone}.

\begin{proposition} \label{prop:sub-dim-zero}
Let $X$ be an irreducible affine $\GL$-variety and let $Z$
be an irreducible closed $\GL$-subvariety. We can then find
$n \ge 0$, a $\GL$-invariant function $h$ on $X$ such that
$h \vert_Z \ne 0$, an elementary $\GL$-variety $Y$, and a
map of $\GL$-varieties $\phi \colon \Sh_n(X)[1/h] \to Y$
such that $\phi \vert_{\Sh_n(Z)[1/h]}$ is dominant, of relative dimension zero, and of invariant finite presentation.
\end{proposition}

\begin{proof}
After shifting and shrinking, we can assume $Z=B \times \bA^{\ulambda}$, with $B$ an
irreducible variety and $\ulambda$ a pure tuple. We thus have $K[Z]=K[B] \otimes_K
\Sym(\bV_{\ulambda})$. Let $x_1, \ldots, x_n$ be algebraically independent elements of
$K[B]$, where $n=\dim(B)$. We have a surjection $K[X] \to K[Z]$. Let $\tilde{x}_i$ be a lift
of $x_i$ to $K[X]^{\GL}$ and choose a map of representations $\bV_{\ulambda} \to K[X]$
lifting the given map $\bV_{\ulambda} \to K[Z]$. Let $R=K[t_1, \ldots, t_n] \otimes
\Sym(\bV_{\ulambda})$. We thus have a map of $\GL$-algebras $R \to K[X]$. We take
$Y=\Spec(R)$ and let $\phi \colon X \to Y$ be the map corresponding to the said algebra
homomorphism. Since the composition $R \to K[X] \to K[Z]$ is injective by construction, it
follows that $\phi \vert_Z$ is dominant. It is clear from the construction that $R \to K[Z]$
is of invariant finite presentation. It is also clear that $Z\{K^n\}$ and $Y\{K^n\}$ have the same dimension for all $n$, and so $\phi \vert_Z$ has relative dimension~0.
\end{proof}

The proposition yields the following more explicit statement:

\begin{corollary}
With notation as in the proposition, there is a commutative diagram
\begin{displaymath}
\xymatrix@C=4em{
\Sh_n(Z)[1/h] \ar[d] \ar@{=}[r] & B \times \bA^{\ulambda} \ar[d]^{p \times \id} \\
\Sh_n(X)[1/h] \ar[r]^{\phi} & C \times \bA^{\ulambda} }
\end{displaymath}
where $\ulambda$ is a pure tuple, $B$ and $C$ are irreducible affine varieties, and $p$ is a dominant generically finite map of varieties.\
\end{corollary}

\begin{proof}
The $\GL$-variety $Y$ from Proposition~\ref{prop:sub-dim-zero} is
elementary, hence of the form $C \times \bA^{\ulambda}$. Now
apply Proposition~\ref{prop:ifp-over-elem} to the map $\phi
\vert_{\Sh_n(Z)[1/h]}: \Sh_n(Z)[1/h] \to C \times \bA^{\ulambda}$ to
conclude that $\Sh_n(Z)[1/h]$ is of the form $B \times \bA^{\ulambda}$
and the map is of the desired form.
\end{proof}

\subsection{Proof of Theorem~\ref{thm:cdone}} \label{ss:cdone-proof}

We now have all the tools in place to prove the theorem. Let $Z \subset X$ be given as in Theorem~\ref{thm:cdone}, i.e., $X$ is an irreducible affine $\GL$-variety and $Z$ is an irreducible closed $\GL$-subvariety of codimension one, in the sense that $\dim(X|Z)=1$. To prove the theorem, we are free to apply the shift functor and replace $X$ and $Z$ with $X[1/h]$ and $Z[1/h]$ where $h$ is an invariant function on $X$ with non-zero restriction to $Z$.

Applying Proposition~\ref{prop:sub-dim-zero}, after shifting and shrinking we can find a dominant map $\phi \colon X \to Y$, where $Y$ is an elementary $\GL$-variety, such that $\phi \vert_Z$ is dominant, of relative dimension zero, and of invariant finite presentation. Since $\dim(X|Z)=1$ and $\dim(Z|Y)=0$, we conclude that $\dim(X|Y)=1$. Applying Proposition~\ref{prop:rel-dim-one}, after shifting and shrinking $Y$ (and replacing $X$ and $Z$ with their base changes over the shrunken $Y$), we can assume that $Y$ is elementary and $X \to Y$ is of invariant finite presentation; this implies $X$ is elementary (Proposition~\ref{prop:ifp-over-elem}). By Proposition~\ref{prop:ifp-trans}, we see that $Z \to X$ is of invariant finite presentation. Since $X$ is elementary, the result now follows from Proposition~\ref{prop:ifp-over-elem}.

\section{Unirationality for pairs} \label{s:uni-pair}

\subsection{Statement of results}

In the previous section, we proved the shift and unirationality theorem for pairs $Z \subset X$ when $Z$ has codimension one. We have not been able to prove the shift theorem for more general pairs. However, we can prove the unirationality theorem in much greater generality; this is the main result of this section:

\begin{theorem} \label{thm:uni-pair}
Let $X$ be an irreducible $\GL$-variety and let $Z$ be an irreducible closed $\GL$-subvariety. Then we can find a commutative diagram
\begin{displaymath}
\xymatrix@C=4em{
B \times \bA^{\ulambda} \ar[r]^{j \times \id} \ar[d]_{\psi} &
C \times \bA^{\ulambda} \ar[d]^{\phi} \\
Z \ar[r] & X }
\end{displaymath}
where $C$ is an irreducible affine variety, $B$ is an irreducible closed subvariety of $C$ of
codimension one, $j \colon B \to C$ is the inclusion, $\ulambda$ is a pure tuple, and $\phi$ and $\psi$ are dominant morphisms of $\GL$-varieties.
\end{theorem}

As a corollary, we obtain the following refinement of the unirationality theorem:

\begin{corollary} \label{cor:uni-pair}
Let $X$ be an irreducible $\GL$-variety and let $x \in X$. Then there exists a dominant map $\phi \colon C \times \bA^{\ulambda} \to X$, with $C$ and $\ulambda$ as in the theorem, such that $x \in \im(\phi)$.
\end{corollary}

\begin{proof}
Apply the theorem with $Z=\ol{O}_x$ the orbit-closure of $x$. Since $\psi$ maps dominantly to $Z$, its image contains a dense open subset of $Z$, and therefore contains $x$ \cite[Proposition 3.4]{bdes}. As $\im(\psi) \subset \im(\phi)$, the result follows.
\end{proof}

The basic idea of the proof of Theorem~\ref{thm:uni-pair} is to blow-up $X$ along $Z$ to reduce to the codimension one case (Corollary~\ref{cor:shift-pair}). Most of the work goes into verifying that blow-ups behave in the desired manner for $\GL$-varieties.

\subsection{Projective \texorpdfstring{$\GL$}{GL}-varieties}

We now establish a result on projective $\GL$-varieties. We make no attempt at generality, and just do what we need for our proof of Theorem~\ref{thm:uni-pair}.

Let $R=\bigoplus_{n \ge 0} R_n$ be a graded $\GL$-algebra. (Every
$\GL$-algebra has the central grading, in which the coordinates on
$\bA^{\lambda}$ have degree $|\lambda|$; 
$R$ is equipped with an additional grading.) Put $S:=R_0$, an ordinary $\GL$-algebra. We make the following assumptions:
\begin{itemize}
\item $R$ is finitely $\GL$-generated (and thus the same is true for $S$).
\item $R$ is a domain (and thus the same is true for $S$).
\item $R_k \ne 0$ for some $k>0$.
\end{itemize}
Let $X=\Proj(R)$ and $Y=\Spec(S)$, and let $\pi \colon X \to Y$ be the natural map. Note that there is no problem in defining $\Proj$ for arbitrary graded rings, and it is always a scheme (see, e.g., \stacks{01M3}), however the map $\pi$ will typically not be a projective morphism (since, by definition, this has a finite type requirement). The following is the result we require:

\begin{proposition} \label{prop:surj}
The map $\pi$ is surjective.
\end{proposition}
\begin{proof}
Let $\tilde{X} = \Spec(R) \setminus V(R_+)$, where $R_+$ is the ideal $\bigoplus_{n>0} R_n$
of $R$. This is the affine cone over $X$, with the origin removed. We have a natural map
$\tilde{X} \to X$. Since $\tilde{X}$ is a quasi-affine $\GL$-variety, the image of $\tilde{X}
\to X \stackrel{\pi}{\to} Y$ is $\GL$-constructible by Theorem~\ref{thm:chev}. Since
$\tilde{X} \to X$ is surjective, it follows that the image of $\pi$ is $\GL$-constructible.
Put $X\{K^n\}:=\Proj(R\{K^n\})$. To show that $\pi$ is surjective, it thus suffices to show that $\pi_n\colon X\{K^n\}\to Y\{K^n\}$ is surjective for all $n \gg 0$.

Let $k$ be such that $R_k \ne 0$ and let $n \gg 0$ be such that $R_k\{K^n\} \ne 0$.  Since $R \to R\{K^n\}$ is a surjection of graded rings, we have a natural closed immersion $X\{K^n\} \to X$. Consider the following commutative diagram:
\begin{displaymath}
\xymatrix{
X\{K^n\} \ar[r] \ar[d]_{\pi_n} & X \ar[d]^{\pi} \\
Y\{K^n\} \ar[r] & Y }
\end{displaymath}
The fiber over the generic point of $Y\{K^n\}$ is the $\Proj$ of the graded ring $\Frac(S\{K^n\}) \otimes_{S\{K^n\}} R\{K^n\}$. Since this is a graded domain over a field with a nonzero element of positive degree, its $\Proj$ is non-empty. It follows that $\pi_n$ is dominant. However, $\pi_n$ is also a projective morphism. Thus $\pi_n$ is surjective, which completes the proof.
\end{proof}

\subsection{Blow-ups}

Let $S$ be a finitely $\GL$-generated $\GL$-algebra and let $I$ be a finitely $\GL$-generated ideal of $S$. We assume that $S$ is a domain and $I \ne 0$. The \emph{blow-up algebra} is $R=\bigoplus_{n \ge 0} I^n$, where $I^0=S$, and the \emph{blow-up} of $Y=\Spec(S)$ along $I$ is $X=\Proj(R)$. Note that $R$ is finitely $\GL$-generated; moreover, $R$ is a subalgebra of $S[t]$, and is therefore a domain. Thus Proposition~\ref{prop:surj} shows that the natural map $X \to Y$ is surjective.

For a nonzero element $f \in I$, let $S[I/f]$ be the subset of $\Frac(S)$ consisting of all
elements of the form $a/f^k$ where $a \in I^k$. One easily verifies that this is a subring of
$\Frac(S)$, and thus a domain. The affine scheme $\Spec(S[I/f])$ is naturally an open subset
of $X$, and these sets (as $f$ varies) form an affine cover of $X$ \stacks{0804}. The extended ideal $IS[I/f]$ is the principal ideal $fS[I/f]$ \stacks{07Z3}. We summarize this discussion in the following proposition.

\begin{proposition} \label{prop:blowup}
Let $f$ be a nonzero element of $I$, and consider $\pi \colon \Spec(S[I/f]) \to Y$.
\begin{enumerate}
\item $\pi$ is dominant.
\item Any $s \in Y$ belongs to the image of $\pi$ for some choice of $f$.
\item We have $\pi^{-1}(V(I))=V(f)$.
\end{enumerate}
\end{proposition}

\subsection{Proof of Theorem~\ref{thm:uni-pair}}

We begin with two lemmas.

\begin{lemma} \label{lem:uni-pair-1}
Let $X$ be an irreducible $\GL$-variety, let $f$ be a non-zero $\GL$-invariant function on $X$, and let $Z$ be an irreducible component of $V(f)$, the zero locus of $f$. Then $Z$ has codimension one in $X$.
\end{lemma}

\begin{proof}
For $n \gg 0$, $Z\{K^n\}$ is an irreducible component of the zero locus of $f$ in $X\{K^n\}$, and thus has codimension one by Krull's principal ideal theorem. The result follows. 
\end{proof}

\begin{lemma} \label{lem:uni-pair-2}
Let $X$ be an irreducible affine $\GL$-variety and let $Z$ be an irreducible closed
$\GL$-subvariety. Then we can find a commutative diagram of maps of $\GL$-varieties 
\begin{displaymath}
\xymatrix{
Z' \ar[r] \ar[d] & X' \ar[d] \\
Z \ar[r] & X }
\end{displaymath}
where $X'$ is an irreducible affine $\GL$-variety, $Z'$ is an irreducible closed $\GL$-subvariety of $X'$ of codimension one, and the vertical maps are dominant.
\end{lemma}

\begin{proof}
Write $X=\Spec(S)$ and $Z=V(I)$, where $I$ is a finitely $\GL$-generated ideal of $S$; this is possible by \cite[Theorem 1]{draisma}, but we note that $I$ may not be radical. Let $f \in I$ be nonzero, let $\tilde{X}=\Spec(S[I/f])$ and let $\pi \colon \tilde{X} \to X$ be the natural map. By Proposition~\ref{prop:blowup}, we may choose $f$ so that the image of $\pi$ contains the generic point of $Z$. Let $n$ be such that $f$ is $G(n)$-invariant. Then $S[I/f]$ is a finitely $G(n)$-generated $G(n)$-algebra, and so $\tilde{X}$ is a $G(n)$-variety. Of course, $\pi$ is a map of $G(n)$-varieties.

Let $\tilde{Z}=\pi^{-1}(Z)$. This is a closed $G(n)$-subvariety of $\tilde{X}$; in fact, it
is the vanishing locus of $f$ in $\tilde{Z}$ by Proposition~\ref{prop:blowup}. Let
$\tilde{Z}_1, \ldots, \tilde{Z}_r$ be the irreducible components of $\tilde{Z}$. Since
$\pi(\tilde{Z})$ contains the generic point of $Z$, it follows that one of these components, say $\tilde{Z}_1$, maps dominantly to $Z$.

We have the following diagram of $G(n)$-varieties
\begin{displaymath}
\xymatrix{
\tilde{Z}_1 \ar[r] \ar[d] & \tilde{X} \ar[d] \\
Z \ar[r] & X }
\end{displaymath}
where the vertical maps are dominant. Let $\GL$ act on all the above objects through the isomorphism $\GL \cong G(n)$. Let $Z'=\tilde{Z}_1$ and $X'=\tilde{X}$, with this $\GL$-action. The above diagram thus becomes
\begin{displaymath}
\xymatrix{
Z' \ar[r] \ar[d] & X' \ar[d] \\
\Sh_n(Z) \ar[r] & \Sh_n(X) }
\end{displaymath}
where all objects are $\GL$-varieties and all arrows are morphisms of $\GL$-varieties. Composing with the natural maps $\Sh_n(X) \to X$ and $\Sh_n(Z) \to Z$ gives the diagram in the statement of the lemma. Since $Z'$ is an irreducible component of the vanishing locus of the $\GL$-invariant function $f$ on $X'$, it has codimension one by Lemma~\ref{lem:uni-pair-1}.
\end{proof}

We now prove the theorem. Let $X$ and $Z$ be as in the statement of the theorem. Choose a diagram as in Lemma~\ref{lem:uni-pair-2}. Now apply Corollary~\ref{cor:shift-pair} to $Z' \subset X'$ to obtain a diagram
\begin{displaymath}
\xymatrix@C=4em{
B \times \bA^{\ulambda} \ar[r]^{j \times \id} \ar[d] &
C \times \bA^{\ulambda} \ar[d] \\
Z' \ar[r] & X' }
\end{displaymath}
where the notation is as in Corollary~\ref{cor:shift-pair}. Composing this square with the one from Lemma~\ref{lem:uni-pair-2} gives the desired result.

\section{Existence of curves} \label{s:curves}

\subsection{The main theorem}

The following is the main theorem of this section:

\begin{theorem} \label{thm:curve}
Let $X$ be an irreducible $\GL$-variety, let $U$ be a
non-empty open subset of $X$, and
let $x$ be an $\Omega$-point of $X$, where $\Omega$ is a field extension of $K$. Then there
is an irreducible smooth curve $C$ over a finite extension $\Omega'$ of $\Omega$ and a map of $K$-schemes $i \colon C \to X$ such that $\im(i)$ meets $U$ and contains $x$.
\end{theorem}

We first require a lemma.

\begin{lemma} \label{lem:elem-curve}
Let $B$ be an irreducible variety, let $\ulambda$ be a tuple, and let $x$ and $y$ be two $\Omega$-points of $B \times \bA^{\ulambda}$. Then there exists an irreducible smooth curve $C$, defined over a finite extension $\Omega'$ of $\Omega$, and a map of $K$-schemes $i \colon C \to B \times \bA^{\ulambda}$ such that $\im(i)$ contains $x$ and $y$.
\end{lemma}

\begin{proof}
Write $x=(x_1,x_2)$ where $x_1$ is a point of $B$ and $x_2$ is a point of $\bA^{\ulambda}$, and similarly write $y=(y_1,y_2)$. Consider the map (defined over $\Omega$)
\begin{displaymath}
\phi \colon B \times \bA^1 \to B \times \bA^{\ulambda}, \qquad
\phi(b, t) = (b, (1-t)x_2+ty_2).
\end{displaymath}
We have $x=\phi(x_1,0)$ and $y=\phi(y_1,1)$. Since $B \times \bA^1$ is an irreducible finite dimensional variety, any two points can be joined by an irreducible smooth curve. Let $j \colon C \to B \times \bA^1$ be a curve joining $(x_1,0)$ and $(y_1,1)$. Then $i=\phi \circ j$ is the curve joining $x$ and $y$.
\end{proof}

\begin{proof}[Proof of Theorem~\ref{thm:curve}]
Applying Corollary~\ref{cor:uni-pair} let $\phi \colon B \times \bA^{\ulambda} \to X$ be a dominant map of $\GL$-varieties with $x \in \im(\phi)$, where $B$ is an irreducible affine variety over $K$. Let $\tilde{x}$ be a point in $B \times \bA^{\ulambda}$ defined over a finite extension of $\Omega$ that maps to $x$. Since $\phi$ is dominant, its image meets $U$, and hence $\phi^{-1}(U)$ is a non-empty open subset of $B \times \bA^{\ulambda}$. Let $\tilde{y}$ be a point of $\phi^{-1}(U)$ defined over a finite extension of $\Omega$. Appealing to the lemma, we can find an irreducible smooth curve $C$ defined over a finite extension of $\Omega$ and a map $j \colon C \to B \times \bA^{\ulambda}$ that joins $\tilde{x}$ and $\tilde{y}$. The curve $i=\phi \circ j$ in $X$ joins $x$ and $U$.
\end{proof}

The following example shows that in infinite dimensional algebraic geometry, curves can be scarce. Thus Theorem~\ref{thm:curve} truly is a special property of $\GL$-varieties.

\begin{example} \label{ex:no-curve}
Let $R=K[x_1,x_2,\ldots]$, let $I$ be the kernel of the map
\begin{displaymath}
K[x_1,x_2,\ldots] \to K[t^q \mid q \in \bQ_{\geq 0}], \qquad x_i \mapsto t^{1/i},
\end{displaymath}
and let $Z=V(I)$ be the closed subvariety of $\Spec(R)=\bA^{\infty}$ defined by $I$.
If $C$ is a curve over any finite extension $\Omega$ of $K$ (which,
we recall, is required to be finite type over $\Omega$) then any map $C \to Z$ is
constant, as one can see by looking at the map of coordinate rings. In
particular, the origin in $Z$ cannot be joint by such a curve to
the open set where $x_1$ is nonzero. 
\end{example}

\subsection{The Krull intersection theorem}

We now prove a version of the Krull intersection theorem for $\GL$-algebras. We begin with the following observation.

\begin{proposition}
Let $X$ be a $\GL$-variety, let $x$ be a point of the scheme $X$, and let $f$ be a function on $X$. Suppose that $f$ vanishes on every curve through $x$, i.e., given an irreducible curve $C$ defined over a field extension $\Omega$ of $K$ and a $K$-morphism $i \colon C \to X$ such that $x \in \im(i)$, we have $i^*(f)=0$. Then $f$ vanishes on every irreducible component of $X$ that contains $x$. In particular, $f$ vanishes on a non-empty open neighborhood of $x$, and $f$ vanishes in the local ring $\cO_{X,x}$.
\end{proposition}

\begin{proof}
It suffices to treat the case where $X$ is irreducible; we must then show $f=0$ identically. Suppose not. Then $X[1/f]$ is non-empty, and so by Theorem~\ref{thm:curve} there is a curve $i \colon C \to X$ such that $x \in \im(i)$ and $\im(i)$ meets $X[1/f]$. But then $i^*(f) \ne 0$, a contradiction.
\end{proof}

The following theorem is our main result.

\begin{theorem}
Let $A$ be a $\GL$-algebra that is reduced and finitely $\GL$-generated, and let $\fp$ be a $\GL$-stable prime ideal of $A$. Then $\bigcap_{n \ge 1} \fp^n A_{\fp}=0$.
\end{theorem}

\begin{proof}
Let $X=\Spec(A)$, and let $x \in X$ be the point $\fp$. Suppose $f$ is an element of $A$ whose image in $A_{\fp}$ belongs to $\bigcap_{n \ge 1} \fp^n A_{\fp}$; we must show that the image of $f$ in $A_{\fp}$ vanishes. Let $i \colon C \to X$ be a curve through $x$ in the above sense. There is some prime ideal $\fq$ of $\Omega[C]$ that contracts to $\fp$ under the ring homomorphism $i^* \colon K[X] \to \Omega[C]$, and so $i^*$ induces a local ring homomorphism $K[X]_{\fp} \to \Omega[C]_{\fq}$. The condition on $f$ implies that its image in $\Omega[C]_{\fq}$ belongs to $\bigcap_{n \ge 1} \fq^n \Omega[C]_{\fq}$, and so $i^*(f)=0$ by the classical Krull intersection theorem. The previous proposition shows that $f$ vanishes in some open neighborhood of $x$, and thus maps to~0 in $A_{\fp}$.
\end{proof}

\section{Image closures} \label{s:image}

\subsection{Bounded Laurent points}

The purpose of \S \ref{s:image} is to prove the first main theorem of this paper, Theorem~\ref{mainthm1}, which states that points in an image closure can be realized by limits. This can be phrased in more abstract terms using bounded points in fields of Laurent series. We now give a precise definition of what this means, and prove some simple results concerning the notion.

\begin{definition}
Let $R$ be a $\GL$-algebra and let $\Omega$ be an extension field of $K$. We say that a $K$-algebra homomorphism $f \colon R \to \Omega\lpp t \rpp$ is \defn{bounded} if for every finite-length $\GL$-subrepresentation $V$ of $R$ there is some $n$ such that $f(V) \subseteq t^{-n} \Omega\lbb t \rbb$. We say that an $\Omega\lpp t \rpp$-point of a quasi-affine $\GL$-variety $X$ is \defn{bounded} if the corresponding homomorphism $\Gamma(X, \cO_X) \to \Omega\lpp t \rpp$ is bounded.
\end{definition}

\begin{example}
Let $v$ be an $\Omega\lpp t \rpp$-point of $\bA^{\ulambda}$. Then $v$ is bounded if and only if it can be expressed in the form $v=\sum_{k \ge n} v_k t^k$ where $n$ is an integer and the $v_k$'s are $\Omega$-points of $\bA^{\ulambda}$. This can be proven using Proposition~\ref{prop:bdd-gens} below.
\end{example}

We now establish a few general properties of bounded homomorphisms and points.

\begin{proposition}
Let $R$ be a $\GL$-algebra, let $f \colon R \to \Omega \lpp t \rpp$ be a $K$-algebra homomorphism and let $m \ge 0$. Then $f$ is bounded if and only if it is bounded when regarding $R$ as a $G(m)$-algebra (or, equivalently, as a homomorphism from the $\GL$-algebra $\Sh_m(R)$).
\end{proposition}
\begin{proof}
Suppose $f$ is bounded when $R$ is thought of as a $\GL$-algebra and let $V$ be a finite-length $G(m)$-subrepresentation of $R$. Then $V$ generates a finite-length $\GL$-subrepresentation $W$. Since $f$ is bounded when $R$ is thought of as a $\GL$-algebra, we have $f(W) \subseteq t^{-n} \Omega \lbb t \rbb$ for some $n$. Since $f(V) \subseteq f(W)$ it follows that $f$ is bounded when $R$ is thought of as a $G(m)$-algebra.

Now suppose that $f$ is bounded when $R$ is thought of as a $G(m)$-algebra and let $V$ be a finite-length $\GL$-subrepresentation of $R$. Then $V$ has finite length as a $G(m)$-representation and so $f(V) \subseteq t^{-n} \Omega \lbb t \rbb$ for some $n$. Thus $f$ is bounded when $R$ is thought of as a $\GL$-algebra.
\end{proof}

\begin{proposition} \label{prop:bdd-gens}
Let $R$ be a $\GL$-algebra generated by a finite-length $\GL$-subrepresentation $V$ and let $f \colon R \to \Omega\lpp t \rpp$ be a $K$-algebra homomorphism. Then $f$ is bounded if and only if $f(V) \subseteq t^{-n} \Omega \lbb t \rbb$ for some $n$.
\end{proposition}
\begin{proof}
If $f$ is bounded then $f(V) \subseteq t^{-n} \Omega \lbb t \rbb$ for some $n$, by definition. Conversely, suppose this condition holds and let $W$ be an arbitrary finite-length $\GL$-subrepresentation of $R$. Let $V^k$ be the image of the natural map $\Sym^k(V) \to R$. Since $V$ generates $R$ and $W$ is finitely generated as a representation, it follows that $W \subseteq \sum_{k=0}^m V^k$ for some integer $m$. We thus see that $f(W) \subseteq t^{-nm} \Omega \lbb t \rbb$, and so $f$ is bounded.
\end{proof}

The next property is deeper, but easily proved using the decomposition theorem from \cite{bdes}.

\begin{proposition} \label{prop:bounded-lift}
Let $\phi \colon X \to Y$ be a map of $\GL$-varieties. Let $\Omega$ be an arbitrary field extension of $K$ and let $y$ be a bounded $\Omega\lpp t \rpp$-point in $\im(\phi)$. Then $y$ lifts to a bounded $\Omega'\lpp t^{1/n} \rpp$-point of $X$ for some $n$ and some finite extension $\Omega'$ of $\Omega$.
\end{proposition}
\begin{proof}
This is clear for elementary maps, and so the general case follows from the decomposition theorem (Theorem~\ref{thm:decomp}).
\end{proof}

\subsection{The main result}

Suppose that $y(t)$ is an $\Omega\lpp t\rpp$-point of an affine scheme $X$ and $x$ is an $\Omega$-point. We write $x=\lim_{t \to 0} y(t)$ to indicate that $y$ extends to an $\Omega\lbb t \rbb$-point of $X$ and that its value at~0 is $x$. We sometimes also use this expression for points $x$ in the topological space $X$.

\begin{theorem} \label{thm:limit}
Let $\phi \colon Y \to X$ be a morphism of $\GL$-varieties
and let $x$ be an $\Omega$-point of $\ol{\im{\phi}}$, where
$\Omega$ is a field extension of $K$. Then there exists a
finite extension $\Omega'$ of $\Omega$ and a bounded
$\Omega'\lpp t \rpp$-point $y(t)$ of $Y$ such that $x=\lim_{t \to 0} \phi(y(t))$.
\end{theorem}


\begin{proof}
Replacing $X$ with $\ol{\im{\phi}}$, we assume $\phi$ is dominant. Let $U$
be a non-empty open $\GL$-stable subset of $X$ contained in $\im{\phi}$;
this exists by Theorem~\ref{thm:chev}.  By Theorem~\ref{thm:curve} there
exists an irreducible smooth curve $C$, defined over a finite extension
$\Omega'$ of $\Omega$, that passes through $x$ and intersects $U$.

Letting $t$ be a uniformizer on $C$ at $x$, the complete
local ring of $C$ at $x$ is isomorphic to $\Omega' \lbb t \rbb$. The composition
\begin{displaymath}
\Spec(\Omega'\lbb t \rbb) \to C \to X
\end{displaymath}
carries the generic point of $\Spec(\Omega'\lbb t \rbb)$ into
$U$ (as it first maps to the generic point of $C$). Hence
this composition is an $\Omega'\lpp t \rpp$-point $x(t)$ of $U$
such that $x=\lim_{t \to 0} x(t)$. 

Now, replacing $t$ with $t^n$ and $\Omega'$ by a finite
extension if necessary, we can lift $x(t)$ to a bounded
$\Omega'\lpp t \rpp$-point $y(t)$ of $Y$ by Proposition~\ref{prop:bounded-lift}. We have $x=\lim_{t \to 0} \phi(y(t))$, as required.
\end{proof}

\section{Uniformity} \label{s:app}

Let $\phi \colon X \to Y$ be a map of $\GL$-varieties.
Theorem~\ref{thm:limit} shows that every point of $\ol{\im(\phi)}$ can
be realized as a limit of a bounded $\Omega\lpp t \rpp$-point. We now
show that one can uniformly control the negative powers of $t$
appearing in such a limit. We will require the following notion: given an embedding $X \to
\bA^{\ulambda}$, we define the \defn{size} of a $\Omega\lpp t
\rpp$-point $x(t)$ of $X$ to be the minimal $s \ge 0$ such that each
coordinate of $t^s x(t)$ in $\bA^{\ulambda}$ is in $\Omega \lbb t
\rbb$. Our main theorem is then:

\begin{theorem} \label{thm:Uniformity}
Let $\phi \colon X \to Y$ be a dominant map of affine $\GL$-varieties, and fix a closed embedding $X \to \bA^{\ulambda}$. Then there exists $s \ge 0$ such that for any point $y \in Y$ there is an $\Omega\lpp t \rpp$-point $x(t)$ of $X$, for some field $\Omega/K$, of size at most $s$ such that $y=\lim_{t \to 0} \phi(x(t))$.
\end{theorem}

If $X' \to X$ is a dominant morphism of $\GL$-varieties, then it suffices
to prove the theorem for the induced map $X' \to Y$. We may thus assume
$X=B \times \bA^{\umu}$ for some affine variety $B$.  Moreover, shrinking
$B$, we may assume $B$ smooth. We make these assumptions in what follows.

The key idea is to construct limits geometrically. A \defn{limiting datum} is a tuple
\begin{displaymath}
\sL=(S,\ol{C},P,Q_{\bullet},\pi,i,r,s)
\end{displaymath}
where
\begin{itemize}
\item $S$ is a smooth finite dimension affine variety.
\item $\ol{C} \to S$ is a smooth projective family of curves with geometrically irreducible fibers.
\item $P$ and $Q_1, \ldots, Q_m$ are sections of $\ol{C} \to S$ that are fiberwise non-intersecting; put $C=\ol{C} \setminus \{Q_1, \ldots, Q_m\}$.
\item $\pi$ is a uniformizer on $C$ at $P$ and nowhere vanishing on $C \setminus P$; in other words, the ideal sheaf of $P$ in $C$ is generated by $\pi$.
\item $i \colon C \setminus P \to B$ is a morphism of varieties.
\item $r,s$ are non-negative integers.
\end{itemize}
Fix such a datum. Define a map of $\GL$-varieties
\begin{displaymath}
\psi \colon (C \setminus P) \times (\bA^{\umu})^{r+s-1} \to Y, \qquad
\psi(x, v_{-s}, \ldots, v_r) = \phi(i(x), \sum_{i=-s}^r \pi(x)^i v_i).
\end{displaymath}
There is a maximal closed subscheme $Z$ of $S \times (\bA^{\umu})^{r+s-1}$ such that $\psi$ extends to $P$ over $Z$. We see this as follows. Given a function $f \colon Y \to \bA^1$, Zariski locally consider the Laurent expansion of the function $\psi^*(f)$ at $P$. The scheme $Z$ is defined as the vanishing locus of the negative coefficients of this series, for all choices of $f$. Since $Z$ is canonically defined, it is $\GL$-stable.

Let $\psi' \colon Z \to Y$ be the function induced by evaluating $\psi$ at $P$. Define $\vert \sL \vert$ to be the image of $\psi'$. This is a $\GL$-constructible subset of $Y$ by Chevalley's theorem (Theorem~\ref{thm:chev}). Moreover, there is some $s'$ such that every point in $\vert \sL \vert$ can be realized by a limit of size $\le s'$. Thus to prove the theorem, it suffices to show that there are limiting data $\sL_1, \ldots, \sL_n$ such that $Y=\bigcup_{i=1}^n \vert \sL_i \vert$. This is implied by the following lemma, by $\GL$-constructibility:

\begin{lemma}
We have $Y=\bigcup \vert \sL \vert$, where the union is over all choices of $\sL$.
\end{lemma}

\begin{proof}
Let $y \in Y$ be given. By Theorem~\ref{thm:limit}, there is a bounded $\Omega \lpp t \rpp$-point $x(t)$ of $X$ such that $y=\lim_{t \to 0} \phi(x(t))$; here $\Omega/K$ is some algebraically closed field. Let $x_1(t)$ be the $B$-component of $x(t)$, and let $x_2(t)$ be the $\bA^{\umu}$-component of $x(t)$. We can find a curve $C_{\Omega}$ over $\Omega$, an $\Omega$-point $P_{\Omega}$ of $C_{\Omega}$, a uniformizer $\pi_{\Omega}$, and a map $i \colon C_{\Omega} \setminus P_{\Omega} \to B_{\Omega}$ such that the map $x'_1(t) \colon \Omega \lpp t \rpp \to B_{\Omega}$ induced by $i$ (with $t=\pi$) agrees with $x_1$ to very high order. A standard argument shows that this data arises as the generic point of similar data $C$, $P$, $\pi$, and $i$ defined over an irreducible variety $S$ over $K$. By shrinking $S$ (i.e., replacing it with a dense open subset), we can arrange for the conditions in the definition of limiting datum. Finally, note that we can truncate $x_2(t)$ to some high degree without changing the limit. We can thus represent $x_2(t)$ as a Laurent polynomial in $\pi$ with coefficients in $\bA^{\umu}$. This completes the proof.
\end{proof}

\section{De-bordering}

In this section, we establish a general de-bordering result for
$\GL$-constructible cones $Z \subseteq \bA^{\ulambda}$, and derive
Theorem~\ref{thm:DeBorderingSymd} from this. As in the introduction,
we assume that $K$ is algebraically closed. 

\subsection{An interpolation result}

Let $\ol{C}$ be an irreducible smooth projective curve over $K$, let
$Q_1,\ldots,Q_m$ be $K$-points on $\ol{C}$, and let $C=\ol{C}
\setminus \{Q_1, \ldots, Q_m\}$. The \emph{degree} of a rational
function on $C$ is the number of zeros it has on $\ol{C}$, counted
with multiplicity. Let $V$ be a vector space, and identify $V^*$ with
the scheme $\Spec(\Sym(V))$; note that $V^*$ is a product of copies
$\bA^1$ indexed by a basis of $V$. For the purpose of this section, we say that a morphism $C \to V^*$ has \defn{degree} $\le d$ if each component does. We require the following interpolation result:

\begin{proposition} \label{prop:interp}
Let $\phi \colon C \to V^*$ be a morphism of degree $\le d$. Let
$r=md+1$ and let $P_1, \ldots, P_r$ be a sufficiently general
$r$-tuple of $K$-points in $C$. Then $\im(\phi)$ is contained in the span of $\phi(P_1), \ldots, \phi(P_r)$.
\end{proposition}

\begin{proof}
Let $L$ be the space of rational functions $\psi$ on $\ol{C}$ such
that each pole of $\psi$ is contained in $\{Q_1, \ldots, Q_m\}$ and
has degree at most $d$. Then $L$ is a vector space of dimension at
most $r$. Let $\psi_1, \ldots, \psi_s$ be a basis of $L$, with $s \le
r$. Each component of $\phi$ is a linear combination of the
$\psi_i$'s, and so we see that $\phi$ has the form $\lambda_1
\psi_1+\cdots+\lambda_s \psi_s$, where the $\lambda_i$'s belong to
$V^*$. For $Q \in C$, let $\psi_Q \in K^s$ be the vector $(\psi_1(Q),
\ldots, \psi_s(Q))$, and let $\lambda=(\lambda_1, \ldots, \lambda_s)$,
so that $\phi(Q)$ is the dot product $\lambda \cdot \psi_Q$. Thus the
image of $\phi$ consists of all dot products $\lambda \cdot \psi_Q$
with $Q \in C$. Since the $P_i$ are sufficiently general, each $\psi_Q$ is a linear combination of the $\psi_{P_i}$'s, and so the result follows.
\end{proof}

\subsection{Closure and sums}

Let $Z$ be a constructible subset of $\bA^{\ulambda}$, and assume
that $Z$ is preserved under scalar multiplication.  Let $mZ$ be the sum
$Z+\cdots+Z$, where there are $m$ copies of $Z$. 

\begin{proposition} \label{prop:sum2}
Let $Z \subset \bA^{\ulambda}$ be as above. Then $\ol{Z} \subset rZ$ for some $r$.
\end{proposition}

\begin{proof}
First suppose that there is some morphism of $\GL$-varieties $\phi \colon X \to \bA^{\ulambda}$ with image $Z$. By the unirationality theorem, we may replace $X$ with $B \times
\bA^{\umu}$ where $B$ is a smooth affine variety. Suppose 
$\sL=(S,\ol{C},P,Q_{\bullet},\pi,i,r,s)$
is a limiting datum for this situation, so that $\vert \sL \vert$ is a subset of $\ol{Z}$. Let $y \in \vert \sL \vert$ be a $K$-point. Then there are $K$-points $s \in S$ and $v_{-s}, \ldots, v_r \in \bA^{\umu}$ such that the map
\begin{displaymath}
\psi \colon C_s \setminus Q_{\bullet} \to \bA^{\ulambda}, \qquad
x \mapsto \phi(i(x), \sum_{i=-s}^r \pi(x)^i v_i).
\end{displaymath}
is defined at $P$ and takes the value $y$ there. The degree of this
map can be bounded independent of the point $s$ and the $v_i$'s. By Proposition~\ref{prop:interp}, there is thus some $r$, depending only on $\sL$, such that $y \in rZ$. Since $y$ is arbitrary, this shows that $\vert \sL \vert \subset rZ$. Since $\ol{Z}$ is covered by finitely many sets of the form $\vert \sL \vert$, the result follows.

We now treat the general case. The conclusion does not depend on $\GL$, and so we are free to shift. After shifting, we can find a dense open affine $\GL$-subvariety $U$ of $\ol{Z}$ that is contained in $Z$. Now apply the previous proposition to the map $U \to \bA^{\ulambda}$. This map has image $U$ and image closure $\ol{Z}$, so we find $\ol{Z} \subset rU$, and thus $\ol{Z} \subset rZ$.
\end{proof}

\subsection{Main result on de-bordering}

Let $Z$ be a $\GL$-constructible subset of $\bA^{\ulambda}$ that is
preserved under scalar multiplication. Let $x \in \bA^{\ulambda}$. Define the \defn{$Z$-rank} of $x$, denoted $\rank_Z(x)$, to be the minimal $r$ such that $x \in rZ$; if no such $r$ exists then we define the rank to be $\infty$. Define the \defn{border $Z$-rank} of $x$, denoted $\brank_Z(x)$, to be the minimal $r$ such that $x \in \ol{rZ}$; again, if no $r$ exists then we use $\infty$. Our main de-bordering result is the following:

\begin{theorem} \label{thm:DeBorderingConstr}
Fix $Z$ as above. Then there is a function $\Phi \colon \bN \to \bN$ such that $\rank_Z(x) \le \Phi(\brank_Z(x))$ whenever $\brank_Z(x)$ is finite.
\end{theorem}

\begin{proof}
Given $r$, Proposition~\ref{prop:sum2} shows that there is some $m(r)$ such that
\begin{displaymath}
\ol{rZ} \subset m(r)rZ.
\end{displaymath}
We thus see that if $\brank_Z(x) \le r$ then $\rank_Z(x) \le m(r)r$. We can thus take $\Phi(r)=m(r) r$.
\end{proof}

\subsection{De-bordering and uniform limits in finite dimensions}

Recall from \S\ref{ss:functorial} that a closed $\GL$-subvariety
$Z \subseteq \bA^{\ulambda}$ gives rise to a contravariant
functor that assigns to finite-dimensional $K$-vector space $V$
a closed subvariety $Z\{V\}$ of $\bA^{\ulambda}\{V\}$. Conversely,
every closed subfunctor of $\bA^{\ulambda}$ arises from a closed
$\GL$-subvariety (see, e.g., \cite[Section 1.4]{draisma}). It follows
that Theorem~\ref{thm:DeBorderingConstr} has the following corollary.

\begin{corollary}
Let $Z$ be a closed subfunctor of the functor $\bA^{\ulambda}$ and assume
that $Z\{V\}$ is preserved under scalar multiplication for every $V$.  Then
there is a function $\Phi \colon \bN \to \bN$ such that $\rank_{Z}(x)
\le \Phi(\brank_{Z})$ whenever $\brank_{Z}(x)$ is finite.
\end{corollary}

This corollary immediately implies Theorem~\ref{thm:DeBorderingSymd}
concerning the rank relative to a functorial closed subset $Z(V) \subseteq
\Sym^d(V)$. To be precise, there $Z$ is covariant, and to apply the
corollary above we need to work with the closed subfunctor $\tilde{Z}$
of $\bA^{(d)}$ that maps $V$ to $Z(V^*) \subseteq \Sym^d(V^*)$, which
we identify with $\bA^{(d)}\{V\}=\Spec(\Sym(\bS_{(d)}V))$ in the natural
manner. In Theorem~\ref{thm:DeBorderingSymd} we did not need to require
that $Z(V)$ is preserved under scalar multiplication between because this
is automatic: for each $t \in K$, $t \cdot \id_V$ preserves $Z(V)$ and acts
via the scalar $t^d$ on $\Sym^d(V)$.

In a similar manner, Theorem~\ref{thm:Uniformity} has the following
consequence in finite dimensions. 

\begin{corollary}
Let $Z$ be a closed subfunctor of the functor $\bA^{\ulambda}$ and assume
that $Z\{V\}$ is preserved under scalar multiplication for every $V$.  Then
for every $r$ there exists a constant $\sigma_{Z,r}$ with the following
property. For any vector space $V$, if $x \in \bA^{\ulambda}\{V\}$
has border $Z$-rank $\leq r$, then there is an expression
\[ x=\lim_{t \to 0} \frac{1}{t^s} \sum_{i=1}^r z_i(t) \]
where $s \leq \sigma_{Z,r}$ and $z_i(t)$ is a $K\lbb t \rbb$-point of
$Z\{V\}$. 
\end{corollary}

\begin{proof}
Regard $Z$ as a closed $\GL$-subvariety of $\bA^{\ulambda}$ and
consider the addition map $\phi:Z^r \to \bA^{\ulambda}$. We may regard
$x$ as a point of $\bA^{\ulambda}\{K^n\} \subseteq \bA^{\ulambda}$
via an isomorphism $K^n \to V$. By Theorem~\ref{thm:Uniformity}, we
have $x=\lim_{t \to 0} \phi(y_1(t),\ldots,y_r(t))$ for some $K\lpp t
\rpp$-point $(y_1(t),\ldots,y_r(t)) \in Z^r \subseteq (\bA^{\ulambda})^r$
whose size is bounded by some $s$ depending on $Z,r$ only. This implies
that
\[ x=\lim_{t \to 0} \phi(y_1(t),\ldots,y_r(t))
= \frac{1}{t^s} \phi(t^s y_1(t),\ldots,t^s y_1(t)), \]
and since $x \in Z\{K^n\}$, this identity is preserved if we replace
each $y_i(t)$ by its image in $Z\{K^n\}$ under the natural projection
$Z \to Z\{K^n\}$. We may therefore take $\sigma_{Z,r}=s$ and
$z_i(t):=t^s y_i(t)$, regarded as a $K\lbb t \rbb$-point of $Z\{V\}$
via the inverse isomorphism $V \to K^n$; here we use that $Z\{V\}$ is
preserved under scalar multiplication. 
\end{proof}


\begin{thebibliography}{BBOV2}

\bibitem[AH]{ah}
Tigran Ananyan, Melvin Hochster. Small subalgebras of polynomial rings and Stillman's conjecture. 
\textit{J.~Amer.~Math.~Soc.} \textbf{33} (2020), no.~1, pp.~2910--309. 
\DOI{10.1090/jams/932} \arxiv{1610.09268v1}

%

\bibitem[BDE]{bde}
Arthur Bik, Jan Draisma, Rob H.~Eggermont. Polynomials and tensors of bounded strength. \textit{Commun. Contemp. Math.} \textbf{21}(7) (2019), paper number 1850062. \arxiv{1805.01816}


\bibitem[BBOV1]{bbov}
Edoardo Ballico, Arthur Bik, Alessandro Oneto, Emanuele Ventura. Strength and slice rank of forms are generically equal. \textit{Israel J.~Math.} \textbf{254} (2023), pp.~275--291.
\DOI{10.1007/s11856-022-2397-0} \arxiv{2102.11549}

\bibitem[BBOV2]{bbov2}
Edoardo Ballico, Arthur Bik, Alessandro Oneto, Emanuele Ventura. The set of forms with bounded strength is not closed. \textit{Compt.~Rend.~Math.} \textbf{360} (2022), pp.~371--380. 
\DOI{10.5802/crmath.302} \arxiv{2012.01237}

\bibitem[BDLZ]{bdlz} Arthur Bik, Jan Draisma, Amichai Lampert, Tamar
Ziegler. Strength and partition rank under limits and field
extensions. In preparation.

\bibitem[BDV]{bdv}
Andreas Blatter, Jan Draisma, Emanuele Ventura. Implicitisation and parameterisation in polynomial functors. \arxiv{2206.01555}

\bibitem[BDES1]{bdes}
Arthur Bik, Jan Draisma, Rob H.~Eggermont, Andrew Snowden. The geometry of polynomial representations. \textit{Int.\ Math.\ Res.\ Not.\ IMRN} (2022). \DOI{10.1093/imrn/rnac220} \arxiv{2105.12621}

\bibitem[BDES2]{bdes2}
Arthur Bik, Jan Draisma, Rob H.~Eggermont, Andrew Snowden. Strong unirationality for $\GL$-varieties. In preparation.

\bibitem[DES]{des} Harm Derksen, Rob H.~Eggermont, Andrew Snowden. Topological noetherianity for cubic polynomials. \textit{Algebra Number Theory} \textbf{11} (2017), no.~9, pp.~2197--2212. \DOI{10.2140/ant.2017.11.2197} \arxiv{1701.01849}

\bibitem[Dr]{draisma}
Jan Draisma. Topological Noetherianity of polynomial functors. \textit{J.\ Amer.\ Math.\ Soc.} \textbf{32} (2019), no.~3, pp.~691--707. \DOI{10.1090/jams/923} \arxiv{1705.01419}


%
%
%

\bibitem[DGIJL]{dgijl}
Pranjal Dutta, Fulvio Gesmundo, Christian Ikenmeyer, Gorav Jindal,
Vladimir Lysikov. De-bordering and Geometric Complexity Theory for
Waring rank and related models. \arxiv{2211.07055}

\bibitem[DS]{DS}
Alessandro Danelon, Andrew Snowden. The singular locus of a $\GL$-variety. In preparation.

\bibitem[Gu]{gu}
Yonghui Guan, Equations for secant varieties of Chow varieties. 
{\em Int. J. Algebra Comput.} \textbf{27} (2017), no.~8, 1087--1111.

\bibitem[KaZ]{kazi}
David Kazhdan, Tamar Ziegler. 
Applications of algebraic combinatorics to algebraic geometry. 
\textit{Indag.\ Math., New Ser.} \textbf{32} (2021), no.~6, pp.~1412--1428. \DOI{10.1016/j.indag.2021.09.002} \\\arxiv{2005.12542}

\bibitem[La]{landsberg}
Joseph M.~Landsberg. \textit{Tensors: geometry and applications.} Graduate Studies in Mathematics 128, American Mathematical Society, Providence, RI, 2012.

%

\bibitem[LL]{ll}
Thomas Lehmkuhl, Thomas Lickteig. On the order of approximation in approximative triadic decompositions of tensors.  \textit{Theor.\ Comput.\ Sci.} \textbf{66}(1) (1989), pp.~1--14. \DOI{10.1016/0304-3975(89)90141-2}

\bibitem[NSS]{sym2noeth}
Rohit Nagpal, Steven V Sam, Andrew Snowden. Noetherianity of some degree two twisted commutative algebras. \textit{Selecta Math.\ (N.S.)} \textbf{22} (2016), no.~2, pp.~913--937. \DOI{10.1007/s00029-015-0205-y} \arxiv{1501.06925}

%
%

\bibitem[Sc]{schmidt}
W.~M.~Schmidt. The density of integer points on homogeneous varieties. \textit{Acta Math.}
\textbf{154} (1985), no.~3--4, pp.~243--296. \DOI{10.1007/BF02392473}

\bibitem[Stacks]{stacks}
Stacks Project. {\tiny\url{http://stacks.math.columbia.edu}} (accessed April, 2023).

\bibitem[Wa]{wadsworth}
Adrian R.~Wadsworth. Hilbert subalgebras of finitely generated algebras. \textit{J.\ Algebra} \textbf{43} (1976), pp.~298--304. \DOI{10.1016/0021-8693(76)90161-7}

\end{thebibliography}
\end{document}